\documentclass[12pt,reqno]{amsart}
\usepackage{xypic}
\usepackage[group-separator={,}]{siunitx}

\usepackage{amssymb}
\usepackage[margin = 1 in]{geometry}
\usepackage{mathtools}
\usepackage{tikz-cd}
\usepackage{microtype}
\usepackage{hyperref}
\usepackage{bookmark}

\newcount\hour \newcount\minute \newcount\minus
\hour=\time \divide \hour by 60
\minute=\time
\minus=\hour\multiply\minus by -60
\advance \minute by \minus
\def\now{%
\number\hour:%
  \ifnum \minute<10 0\fi%
  \number\minute%
}

\newtheorem{theorem}{Theorem}[section]
\newtheorem{lemma}[theorem]{Lemma}

\newtheorem{corollary}[theorem]{Corollary}

\theoremstyle{definition}
\newtheorem{definition}[theorem]{Definition}
\newtheorem{cnj}[theorem]{Conjecture}
\newtheorem{example}[theorem]{Example}
\newtheorem{algorithm}[theorem]{Algorithm}
\newtheorem{notation}[theorem]{Notation}
\theoremstyle{remark}
\newtheorem{remark}[theorem]{Remark}

\newcommand{\Gammabar}{\bar{\Gamma}}

\newcommand{\mat}[1]{\begin{bmatrix} #1 \end{bmatrix}}
\newcommand{\trans}[1]{\prescript{t}{}{#1}}
\DeclarePairedDelimiterX\set[1]{\lbrace}{\rbrace}{\def\given{\;\delimsize\vert\;}#1}
\DeclarePairedDelimiter{\gen}{\langle}{\rangle}
\DeclarePairedDelimiter{\abs}{\lvert}{\rvert}

\newcommand{\Z}{\mathbb{Z}}
\newcommand{\Xbar}{\overline{X}}
\newcommand{\I}{\mathbb{I}}
\newcommand{\Q}{\mathbb{Q}}
\newcommand{\R}{\mathbb{R}}
\newcommand{\C}{\mathbb{C}}
\newcommand{\B}{\partial X}
\newcommand{\HH}{{\mathcal H}}

\DeclareMathOperator{\frob}{Frob}

\newcommand{\PP}{\mathbb{P}}

\newcommand{\cP}{\mathcal{P}}
\newcommand{\cS}{\mathcal{S}}
\newcommand{\inv}{^{-1}}

\def\O{{\mathcal O}}
\newcommand{\FF}{\mathcal{F}}

\DeclareMathOperator{\Coker}{Coker}

\newcommand{\cusp}{\mathrm{cusp}}

\DeclareMathOperator{\GL}{GL}
\DeclareMathOperator{\SL}{SL}
\DeclareMathOperator{\diag}{diag}

\DeclareMathOperator{\image}{Im}

\DeclareMathOperator{\St}{St}

\DeclareMathOperator{\Ind}{Ind}

\newcommand{\st}{\St}

\begin{document}
\title[Cohomology of cong. subgroups, Steinberg modules, and real quad. fields]{Cohomology of 
 congruence subgroups of $\SL_3(\Z)$,
Steinberg modules, and real quadratic fields}
\author{Avner Ash}
\address{Boston College, Chestnut Hill, MA  02467}
\email{Avner.Ash@bc.edu}
\author{Dan Yasaki}
\address{UNCG, Greensboro, NC  27412}
\email{d\_yasaki@uncg.edu}
\date{\today~\now}

\begin{abstract}
We investigate the homology of a congruence subgroup 
$\Gamma$ of $\SL_3(\Z)$ with coefficients in the 
Steinberg modules $\St(\Q^3)$ and $\St(E^3)$, where $E$ is a real quadratic field and the coefficients are $\Q$. 
By Borel-Serre duality, $H_0(\Gamma, \St(\Q^3))$ is isomorphic to 
$H^3(\Gamma,\Q)$. Taking the image of the connecting homomorphism 
 $H_1(\Gamma, \St(E^3)/\St(\Q^3))\to H_0(\Gamma, \St(\Q^3))$, followed by the Borel-Serre isomorphism,
we obtain a naturally defined Hecke-stable subspace $H(\Gamma,E)$ of 
$H^3(\Gamma,\Q)$.  We conjecture that  $H(\Gamma,E)$ is independent of $E$ and consists of the cuspidal cohomology $H_\cusp^3(\Gamma,\Q)$ plus a certain subspace of $H^3(\Gamma, \Q)$ that is isomorphic to the sum of the cuspidal cohomologies of the maximal faces of the Borel-Serre boundary.  

 We report on computer calculations of $H(\Gamma,E)$ for various
$\Gamma$, $E$ which provide evidence for the conjecture.  We give a partial heuristic for the conjecture.
 \end{abstract}
\maketitle

\section{Introduction}\label{intro}

Let $\Gamma$ be a congruence subgroup of $\SL_n(\Z)$.   For $K$ an extension of 
$\Q$, denote the Steinberg module 
with $\Q$-coefficients of the vector space $K^n$ by $\St(K^n)$.  When $K=\Q$
and 
$n$ is understood from the context, set $\st=\St(\Q^n)$. 
By filtering the inclusion $\st \subset \St(K^n)$ as in~\cite{A}, we can compare the homology groups
$H_*(\Gamma,\st)$ and  $H_*(\Gamma,\St(K^n))$.

Now specialize to $n=2$ or $3$
and $E$ a real quadratic field.  In these cases, the filtration has only one step, and the comparison boils down to the connecting homomorphisms
$\partial\colon H_k(\Gamma,C)\to H_{k-1}(\Gamma,\st)$, where $C=\St(E^n)/ \st$.  It turns out that we can describe $C$ in terms of $\O_E^\times$, the units in the ring of integers of $E$.

In~\cite{AY} we studied the case $n=2$ and $k=1$.  In this paper we study  
$n=3$ and $k=1$.   We are interested in the image of $\partial$, which we call $H(\Gamma,E)$. By Borel-Serre duality we identify $H_{0}(\Gamma,\st)$ with $H^\nu(\Gamma,\Q)$, where $\nu=1$ if $n=2$, and $\nu=3$ if $n=3$.
Then $H(\Gamma,E)$ is the span of a set of modular symbols built in a certain way 
from those $\gamma\in\Gamma$ which are images of elements of $\O_E^\times$
embedded into $\Gamma$.
The units of $E$ are reflected in the 
cohomology of $\Gamma$.

When $n = 2$ it is obvious that $H(\Gamma,E)$ is contained in the cuspidal cohomology $H^1_\cusp(\Gamma,\Q)$, and in~\cite{AY} we conjectured that it is equal to $H^1_\cusp(\Gamma,\Q)$.  We presented numerical evidence for the conjecture, and we also proved the conjecture under the assumption of the Generalized Riemann Hypothesis.

When $n=3$ it is no longer obvious what the relationship should be between $H(\Gamma,E)$ and cuspidal cohomology.  After some numerical experimentation, we conjectured that  $H(\Gamma,E)$ is the sum of 
$H^3_\cusp(\Gamma,\Q)$ and a subspace of $H^3(\Gamma,\Q)$ that maps isomorphically to the part of the cohomology of the Borel-Serre boundary of the locally symmetric space for 
$\Gamma$ that consists of cuspidal cohomology of the maximal boundary faces.  We then tested the conjecture with further computation, without finding any counter-examples.  
Thus $H(\Gamma,E)$ provides an interesting way of constructing cuspidal cohomology, which we hope may enable 
progress in the study of $H^3_\cusp(\Gamma,\Q)$.
\footnote{We plan to study the analogue of $H(\Gamma,E)$ for other fields $E$ in future work.}

The subspace $H(\Gamma,E)$ is stable under the action of the Hecke operators.  
We use this fact as a check on our computations:  Our computations identify a candidate space $H$ for 
$H(\Gamma,E)$, but which we cannot prove is actually all of $H(\Gamma,E)$ unless $H$ happens to be all of $H^3(\Gamma,\Q)$.  However, we do verify that $H$ is Hecke stable,
\footnote{Of course, in practice we can only verify this for a finite number of Hecke operators} and this gives the check.

If we tensor with $\C$, $H(\Gamma,E)\otimes_\Q\C$ can be viewed as a subspace of arithmetic cohomology, again via the 
the Borel-Serre duality isomorphism:  
\[
H_0(\Gamma,\St(\Q^3)\otimes_\Q\C)\xrightarrow{\sim} H^{3}(\Gamma,\C).
\]
Thus $H(\Gamma,E)$ is connected to automorphic representations.  We use the Hecke action computationally to identify the 
automorphic constituents of $H(\Gamma,E)$.
 
The rest of this introduction summarizes in more detail the contents of the paper.  First we define the congruence subgroups, for which we have carried out our computations.

\begin{definition}\label{N1}  Let $N$ be a positive integer.  Writing matrices in block notation with respect to the partition $(1,n-1)$, define
 \begin{itemize}
 \item $\Gamma_0(N,n)^\pm$ to be the subgroup of $\GL_n(\Z)$
consisting of matrices congruent to 
\[
\mat{
*&*\\
0&*
} \pmod N.
\]  
\item $\Gamma_0(N,n)=\Gamma_0(N,n)^\pm \cap \SL_n(\Z)$. 
 \end{itemize}
 If $n$ is understood from the context, we suppress it from the notation.
\end{definition}  

In this paper we take  $\Gamma=\Gamma_0(N,3)$.  When we do the actual computations, we run them for $\Gamma_0(N,3)^\pm$ because that speeds them up.  This is not a problem,  because 
$H_i(\Gamma_0(N,3),\st)$ is isomorphic to $H_i(\Gamma_0(N,3)^\pm,\st)$ for all $i$.
 \footnote{
Lemma:
Let $n$ be odd integer, and let $R$ be a ring in which $2$ is invertible.
Let $M$ be an $R[\Gamma_0(N,n)^\pm]$-module, and suppose that $-I$ acts trivially on $M$.  Then the inclusion of $\Gamma_0(N,n)$ into 
$\Gamma_0(N,n)^\pm$ induces an isomorphism
$H_i(\Gamma_0(N,n),M)\simeq H_i(\Gamma_0(N,n)^\pm,M)$ for all $i$.

Proof:
$\Gamma_0(N,n)^\pm$ contains $\Gamma_0(N,n)$ as a subgroup of index two and is generated by $\Gamma_0(N,n)$ and $-I$.   Let $J$ be the group $\{I,-I\}$.
Because 2 is invertible, transfer shows that 
$H_i(\Gamma_0(N,n),M)\simeq H_i(\Gamma_0(N,n)^\pm,M)_J$.  But $J$ is central so it acts trivially on $H_i(\Gamma_0(N,n)^\pm,M)$.
}
We perform computations to find $H(\Gamma,E)$ for
\begin{itemize}
\item  all $N\le 50$, all prime $N$ with $51\le N\le 100$, and for $N=11^2$ and $13^2$.
\item for the real quadratic fields $E=\Q(\sqrt\Delta)$ with 
 squarefree $\Delta\le 10$.
\end{itemize}

On the basis of some of these computations, we make a conjecture about $H(\Gamma,E)$, which is borne out by the rest of our computations.
A numerical consequence of the conjecture is that the codimension of $H(\Gamma,E)$ in $H^3(\Gamma,\Q)$
should equal $b=\dim H_1(T/\Gamma)$, where $T$ is the  
Tits building  of $\GL_3(\Q)$ (Definition~\ref{tits}).  This was the first thing we noticed in our computational results. 
 
To state the full conjecture, we introduce some objects which are explained in greater detail in Sections~\ref{bound} and~\ref{conj}. 
Let $\cS$ be the symmetric space for $\Gamma$, and let $X$ be the Borel-Serre compactification of the locally symmetric space 
$\Gamma\backslash \cS$.  There is a natural isomorphism $H^*(\Gamma,\Q)\simeq H^*(X,\Q)$.  The interior cohomology $H_!^*(\Gamma,\Q)$ of $\Gamma$ is defined to be what corresponds under this isomorphism to the kernel of the restriction map $H^*(X,\Q)\to H^*(\B,\Q)$  
\footnote{Tensored with $\C$ we obtain the cuspidal cohomology of $\Gamma$:  
$H_!^*(\Gamma,\Q)\otimes_\Q \C= H_\cusp^*(\Gamma,\C)$.  This is not true in general if $n\ge4$.}.
  Lee and Schwermer~\cite{lee-schwermer} show that 
  $H^3(\Gamma,\Q)$ is the direct sum of three pieces,  $H_!^*(\Gamma,\Q)$, $A$, and $B$.  The subspace $A$ comes from the cuspidal cohomology of the maximal faces of $\B$ and is describable in terms of holomorphic cuspforms of weight 2 for subgroups of $\SL_2(\Z)$.  The subspace $B$ has dimension equal to $b$, defined in the paragraph above.

\begin{cnj}\label{conj-intro}
$H(\Gamma,E)= H_!^*(\Gamma,\Q)+A$. 
\end{cnj}

We do not have a good explanation for why the conjecture should be true.  In Section~\ref{conj} we give a heuristic argument for why $H(\Gamma,E)$ should contain $A$.

Our computations are not definitive as long as
$H(\Gamma,E) \ne H^3(\Gamma, \Q)$, because there is the possibility that further computation with additional real quadratic units could conceivably discover more elements in $H(\Gamma,E)$.  The facts that the space we compute to be putatively equal $H(\Gamma,E)$ is Hecke-stable and that the computations agree with our conjecture give us confidence in the results reported here. 
       
Section~\ref{mod} recalls basic facts about the Steinberg module and about modular symbols.  
Section~\ref{conn} derives a formula for the connecting homomorphism $\partial$ in terms of modular symbols.   
Section~\ref{stuff} determines the elements of $\Gamma$ corresponding to units in the ring of integers of $E$. 
Section~\ref{2} describes, with proof,  an algorithm for computing the image of $\partial$.

In Section~\ref{cusps}, we offer a discussion of the $\Gamma$-orbits of 
parabolic subgroups of $\GL_3(\Q)$, which is necessary to determine the faces of
 the Borel-Serre boundary $\B$.  Also we compute the Euler characteristic of the Tits building modulo $\Gamma$, which is the nerve of $\B$.
Section~\ref{bound} describes $\B$ in detail, following~\cite{lee-schwermer}.  

The conjecture itself is contained and discussed in Section~\ref{conj}.

In Section~\ref{hecke} we review the Hecke operators.
Section~\ref{what} describes our method of computation, which uses the Voronoi cellulation of the symmetric space.  
We also prove in Section~\ref{what} a result of independent interest: for $n\le 4$, the Voronoi homology in its lowest degree is isomorphic to $H_0(\Gamma,\st)$.  In fact,  we prove this for the Voronoi homology and the Steinberg module with  coefficients in a general ring $R$ (not just $\Q$).  For this reason, we define the Steinberg module in  Section~\ref{mod} over a general ring $R$.

Section~\ref{comp} explains the methods used in our computations, and 
Section~\ref{results} discusses the computational results.

 Thanks to P.~Gunnells for helpful comments.  Also to D.~Doud and independently to M.~Masdeu for informing us of relevant computations of theirs.

\section{ The Steinberg module and Steinberg homology}\label{mod}

For more information about the Steinberg module than is given here, see the introduction to~\cite{APS} and its references.

Let $K$ be a field, $R$ a ring, and $n\ge2$ an integer.  Let $K^n$ be the vector space of column vectors.  

\begin{definition}\label{tits}
The Tits building $T(K^n)$ is the simplicial complex with one vertex for each subvector space $V\subset K^n$ with $0\ne V\ne K^n$, where the vertices $V_1,V_2,\dots,V_k$ span a simplex if and only if they can be arranged into a flag. 
\end{definition}

\begin{definition}\label{stein}
The Steinberg module 
$\St(K^n;R)$ is the  reduced homology of the Tits building:
\[
\St(K^n;R)=\widetilde H_{n-2}(T(K^n),R).
\]
\end{definition}

Since $\GL_n(K)$ acts on the Tits building, it also acts on $\St(K^n;R)$, making 
the Steinberg module a left-module for the group ring $R\GL_n(K)$. The Steinberg module 
$\St(K^n;\Z)$ is a free $\Z$-module, and $\St(K^n;R)\simeq\St(K^n;\Z)\otimes_\Z R$.

\begin{definition}
Let $\set{v_1,v_2, \dots,v_n}$ be a basis for $K^n$.
The modular symbol $[v_1,v_2,\dots,v_n]$ denotes the element in $\St(K^n;R)$ which is the
  fundamental class of the $(n-2)$-sphere whose vertices are the subspaces of $K^n$ generated by the proper non-empty subsets of  $\set{v_1,v_2,\dots,v_n}$.
  We may and do fix orientations on these spheres in such a way that 
  the action of an element $g\in \GL_n(K)$ on the 
symbol $[v_1,v_2,\dots,v_n]$ satisfies
\[
g[v_1,v_2,\dots,v_n]=[gv_1,gv_2,\dots,gv_n].
\]
 We extend the notation to all $n$-tuples of vectors by setting
$[w_1,w_2,\dots,w_n]=0$ when $w_1,w_2,\dots,w_n$ are linearly dependent vectors in $K^n$.
If $m$ is the matrix with columns $a_1,a_2,\dots,a_n$, we write
\[
[m]=[a_1,a_2,\dots,a_n],
\]
so that $g[m]=[gm]$ for any $g\in \GL_n(K)$. 
\end{definition}

We recall some standard facts about modular symbols and the Steinberg module.  See 
\cite{ash-rudolph}.

\begin{theorem}\label{thm:mod-props} 
Let $K$ be any field.
\begin{enumerate}
\item \label{it:gens} As abelian group, $\St(K^n;\Z)$ 
is generated by $[v_1,v_2,\dots,v_n]$ as $v_1,v_2,\dots,v_n$ range over
  all elements of $K^n$.
\item  \label{it:relations} The following relations hold:
\begin{enumerate}
\item  $[v_1,v_2,\dots,v_n]=0$ if  $v_1,v_2,\dots,v_n$ do not span $K^n$.
\item $[v_1,v_2,\dots,v_n]=[kv_1,v_2,\dots,v_n]$ for any nonzero $k\in K$;
\item $[v_1,v_2,\dots,v_n]=(-1)^s[v_{s(1)},v_{s(2)}\dots,v_{s(n)}]$ for any permutation $s \in S_n$;
\item $[v_1,v_2,\dots,v_n]=[x,v_2,\dots,v_n]+\cdots+
[v_1,\dots,v_{i-1},x,v_{i+1},\dots, v_n]+\newline\quad \dots+[v_1,v_2,\dots,v_{n-1},x]$ for any nonzero $x\in K^n$.
\end{enumerate}

\item\label{it:unipotent}  $\St(K^n;\Z)$ has as a 
free $\Z$-basis the symbols $[u]$, where $u$ runs over all upper triangular unipotent matrices 
in $\GL_n(K)$. 
\end{enumerate} 
\end{theorem}

We call the fourth relation ``passing through $x$''.
 
We need the following theorem in the form stated, in order to apply it to the Voronoi cellulation.  It differs only slightly from an equivalent theorem proved by Bykovskii \cite[Theorem~1]{By} and follows immediately from it.   See~\cite{CFP} for a different proof of this theorem and also related results for the Steinberg module of vector spaces over fields other than $\Q$.

\begin{theorem}\label{by}   The Steinberg module 
 $\St(\Q^n;\Z)$ is isomorphic to the quotient of the free abelian group generated by symbols $\gen{a_1,a_2,\dots,a_n}$ for all ordered $\Z$-bases 
 $\set{a_1,a_2,\dots,a_n}$ 
 of $\Z^n$ modulo the following relations:
  \begin{enumerate}
  \item  $\gen{a_1,a_2,\dots,a_n}=(-1)^s\gen{a_{s(1)},a_{s(2)},\dots,a_{s(n)}}$ 
    for any permutation $s \in S_n$;
    \item $ \gen{-a_1,a_2,\dots,a_n}=\gen{a_1,a_2,\dots,a_n}$, for all $\Z$-bases 
  $\set{a_1,a_2,\dots,a_n}$ of $\Z^n$;
  \item  $\gen{a,b,a_3,\dots,a_n}+\gen{-b,a+b,a_3,\dots,a_n}+\gen{a+b,-a,a_3,\dots,a_n}=0$, for all $\Z$-bases 
  $\set{a,b,a_3,\dots,a_n}$ of $\Z^n$.
  \end{enumerate}
  The isomorphism is given by $\gen{a_1,a_2,\dots,a_n} \mapsto [a_1,a_2,\dots,a_n]$. \end{theorem}
 
\begin{definition}
If $\Gamma$ is any subgroup of $\GL_n(\Q)$, we define the Steinberg homology of 
$\Gamma$ over $R$ to be
 $H_*(\Gamma, \St(\Q^n; R))$.
\end{definition}

 The zero-th Steinberg homology group, $H_0(\Gamma, \St(\Q^n; R))$, will be identified with the group of co-invariants $\St(\Q^n; R)_\Gamma$. 
\begin{definition}
If $\Gamma$ is a subgroup of $\GL_n(\Q)$, and $m\in M_n(\Q)$, we denote the image of 
 $[m]$ in $\St(\Q^n; R)_\Gamma$ by $[m]_\Gamma$.
\end{definition}

  \begin{corollary}
  The Steinberg homology $H_0(\Gamma, \St(\Q^n;R))$ is isomorphic to the free $R$-module
 generated by symbols $\gen{a_1,a_2,\dots,a_n}_\Gamma$
  for all ordered $\Z$-bases $\set{a_1,a_2,\dots,a_n}$ 
 of $\Z^n$ modulo the following relations:
  \begin{enumerate}
  \item  $\gen{a_1,a_2,\dots,a_n}_\Gamma=(-1)^s\gen{a_{s(1)},a_{s(2)},\dots,a_{s(n)}}_\Gamma$ 
    for any permutation $s \in S_n$;
    \item $ \gen{-a_1,a_2,\dots,a_n}_\Gamma=\gen{a_1,a_2,\dots,a_n}_\Gamma$, for all $\Z$-bases 
  $\set{a_1,a_2,\dots,a_n}$ of $\Z^n$;
  \item  $\gen{a,b,a_3,\dots,a_n}_\Gamma+\gen{-b,a+b,a_3,\dots,a_n}_\Gamma+\gen{a+b,-a,a_3,\dots,a_n}_\Gamma=0$, for all $\Z$-bases 
  $\set{a,b,a_3,\dots,a_n}$ of $\Z^n$;
  \item $ \gen{a_1,a_2,\dots,a_n}_\Gamma=\gen{\gamma a_1,\gamma a_2,\dots, \gamma a_n}_\Gamma$, for all $\Z$-bases 
  $\set{a_1,a_2,\dots,a_n}$ of $\Z^n$ and all $\gamma\in\Gamma$.
  \end{enumerate}
\end{corollary}

One reason for the importance of the Steinberg module is the Borel-Serre duality theorem~\cite[Theorem~11.4.2.]{BS}.
We quote the case of it needed for this paper:
\begin{theorem}\label{BS}
Let $\Gamma$ be a subgroup of finite index in $\SL_n(\Z)$, and let $k$ be a field whose characteristic is prime to the order of all torsion elements of $\Gamma$.  Then for any $i$, there is an isomorphism
\[
\lambda \colon H_i(\Gamma, \St(\Q^n;k)) \to H^{\nu-i}(\Gamma, k),
\]
 where $\nu=n(n-1)/2$.
\end{theorem}

\section{The connecting homomorphism} \label{conn}

Let $G$ be a subgroup of $\GL_n(\Q)$, and let $s\in G$.  
Let $M$ be a left $\Q G$-module, and 
let $m\in M^S$.

Let $F_\bullet\to\Q$ be the standard resolution of the group $G$ by free 
$\Q G$-modules.  So $F_i$ is the $\Q$-vector space with 
 basis $(g_0,g_1,\dots,g_i)\in G^{i+1}$, and the action is given by
 \[g(g_0,g_1,\dots,g_i)=(gg_0,gg_1,\dots,gg_i).\]
 We also use the ``bar'' notation, 
 $[h_1|h_2|\cdots|h_i]=(1,h_1,h_1h_2,\dots,h_1h_2\cdots h_i)$.

In this paper, we only need to deal with $1$-cycles and their boundaries.  The cycles we require are of the form $z=(1,s)\otimes_G m=[s]\otimes_G m$.  The boundary of 
$z$ is $((s)-(1))\otimes_G m = (1)\otimes_G (s\inv-1)m$, which vanishes since $sm=m$.  
Let $\partial$ denote the connecting homomorphism in the long exact sequence of homology derived from the short exact sequence of left $\Q G$-modules
\[
0\to D\to N\to M \to 0.
\]
We compute $\partial(z)$ as follows:  Lift $m$ to $n\in N$ so that $s\inv n=n+d$, for some $d\in D$.  
Lift the cycle to the chain $[s]\otimes_G n$.  The boundary of this chain is a cycle with coefficients in $D$, and its image in $H_0(G,D)$ 
 is equal to $\partial(z)$.  So
\[\partial(z)=  
(1)\otimes_G (s\inv-1)n = (1)\otimes_G d\in D_G.\]  

\section{Unital matrices}\label{stuff}

The algorithm to be developed in Section~\ref{2}  requires us to begin with a supply of ``unital'' matrices.  We define them in this section.  If $v$ is a vector in affine space, 
$\hat v$ denotes the image of $v$ in the corresponding projective space.

First, let $n$ to be any positive integer, and let $E/\Q$ be a finite extension.  
Let $b\in E^n\setminus\Q^n$, and assume that the entries of $b$ span
$E$ over $\Q$. Suppose $g\in\GL_n(\Q)$ stabilizes $\hat b$ 
in $\PP(E^n)$.  Then there exists $x(g)\in E^\times$ such that 
\[
g b = x(g) b.
\]
The map $g\mapsto x(g)$ is a homomorphism from the stabilizer of 
$\hat b$ in $\GL_n(\Q)$ to $E^\times$.  The minimal polynomial of $x(g)$ over $\Q$ divides the characteristic polynomial of $g$. Therefore, if $g$ is also in
$\GL_n(\Z)$, then $x(g)$ must be a unit in $\O_E^\times.$

\begin{definition}\label{unital} Let $n=2$.
Let $\beta\in E\setminus\Q$.
We say that $\gamma\in\GL_2(\Z)$ is \emph{$\beta$-unital} if 
$\gamma(\beta:1)=(\beta:1)$.  If $\gamma$ is 
$\beta$-unital for some $\beta\in E\setminus\Q$, we say $\gamma$ is \emph{unital}.
\end{definition}

If $n=2$ and $E$ is a real quadratic field, then $\gamma$ is $\beta$-unital if and only if there exists an integer $k$ such that 
\[
\gamma\begin{bmatrix}
\beta\\
1
\end{bmatrix}
=\pm\epsilon^k\begin{bmatrix}
\beta\\
1
\end{bmatrix},
\]
where $\epsilon$ is the fundamental unit of $E$.

\section{\texorpdfstring{$H(\Gamma,E)$}{H(Gamma,E)}}\label{2}

Let $E$ be a real quadratic field, and let $e$ be the column vector
$\trans{(1,0,0)}$.

For any $n\ge2$, from~\cite{A} we have a filtration of  
$\St(E^n)$ that is stable 
under the natural action of $\GL_n(\Q)$:
\[
0\subset \st= \FF_0 \subseteq \FF_1 \subseteq \cdots \subseteq \FF_n= \St(E^n).
\]
\begin{definition}
$\FF_m$ is the $\Q$-span of all modular symbols $[a_1,\dots,a_{n-m},b_1,\dots,b_m]\in \St(E^n)$ where $a_i\in \Q^n$ for all $i$ and $b_j\in E^n$ for all $j$.
\end{definition}

Now set $n=3$.
From \cite[Section~6]{A}, we know the following:
\begin{itemize}
\item It turns out that  $\FF_1=\FF_2=\FF_3$, so that the filtration has only one real step:  
$\st=\FF_0\subset\FF_1=\St(E^3)$.  We denote the image of a modular symbol $[x,y,z]\in\FF_1$ in 
$\FF_1/\FF_0$ by $[x,y,z]'$.

\item The quotient $C=\FF_1/\FF_0$ is isomorphic to a direct sum of induced modules $\I(a,b)$, for $b\in\Omega$ and $a\in A(b)$, where:
we choose a set of representatives $\Omega'$ of the $\SL_3(\Z)$-orbits  of 
$\PP^2(E)\setminus\PP^2(\Q)$; for each $b'\in\Omega'$ we choose a nonzero $b\in E^3$ such that $\hat b=b'$ and let $\Omega$ be the set of these $b$'s; for each $b$, $\SL_3(\Z)_b$ denotes the stabilizer in $\SL_3(\Z)$ of $b'$;
let $A(b)'$ be a set of representatives of the $\SL_3(\Z)_b$-orbits 
of $\PP^2(\Q)$; for each $a'\in A(b)'$ we choose a nonzero 
$a\in\Q^3$ such that $\hat a = a'$ 
and let $A(b)$ be the set of these $a$'s; 
$\SL_3(\Z)_{a,b}$ denotes the stabilizer in $\SL_3(\Z)_b$ of $a'$;
and
\[
\I(a,b)=\Ind(\SL_3(\Z)_{a,b},\SL_3(\Z),\Q_{a,b}),
\]
where $\Q_{a,b}$ is the 1-dimensional $\Q$-vector space spanned by 
  $[e,a,b]$.
  
\item
In the isomorphism of  $\SL_3(\Z)$-modules
\[
C\simeq\bigoplus_{b\in\Omega}  \bigoplus_{a\in A(b)} \I(a,b)
\]
we may assume that each $b$ has the form
$b=\trans{(\beta, 1, 0)}$
for some $\beta\in E\setminus\Q$, and
each $a$ has the form
$a=\trans{(a_1,a_2,1)}$
for some $a_1, a_2\in\Q$.
\end{itemize}

To determine $C$ as a $\Gamma$-module,
restrict each $\I(a,b)$ to $\Gamma$ and use the double coset formula for the restriction of an induced module (e.g., \cite[page 69]{B}).
We obtain the following theorem:
\begin{theorem}
As a $\Gamma$-module, 
\[
 C \simeq \bigoplus_{b\in\Omega}  \bigoplus_{a\in A(b)} \bigoplus_{d\in E(a,b)} \I(a,b,d),
\]
  where
   $E(a,b)$ is a set of representatives of the double cosets 
  $\Gamma\backslash \SL_3(\Z)/ \SL_3(\Z)_{a,b}$ and
\[
\I(a,b,d)=\Ind(\Gamma\cap d\SL_3(\Z)_{a,b}d\inv,\Gamma,d\Q_{a,b}).
\]
Here, $d\Q_{a,b}$ denotes the module $\Q_{a,b}$ where $dgd\inv\in d\SL_3(\Z)_{a,b}d\inv$ acts via the formula
$dgd\inv(r)=gr$.  
  \end{theorem}

Trace back $d\Q_{a,b}$ via the isomorphism
$
 C \simeq \bigoplus_{b\in\Omega}  \bigoplus_{a\in A(b)} \bigoplus_{d\in E(a,b)} \I(a,b,d)
  $ 
  in order to view it as a subspace of $C$.  When we do that, we find
that it becomes the
$1$-dimensional $\Q$-vector space spanned by 
  $m=d[e,b,a]'$, and the action of $dgd\inv\in d\SL_3(\Z)_{a,b}d\inv$ is given by 
  $dgd\inv \cdot m=dgm$ (which is the way $dgd\inv$ naturally acts on $d[e,b,a]'$.)  This proves:
  
  \begin{corollary}\label{c}
As a $\Gamma$-module, 
  \[
 C= \bigoplus_{b\in\Omega}  \bigoplus_{a\in A(b)} \bigoplus_{d\in E(a,b)} \I'(a,b,d),
  \] 
  where
   $E(a,b)$ is a set of representatives of the double cosets 
  $\Gamma\backslash \SL_3(\Z)/ \SL_3(\Z)_{a,b}$ and
  \[
\I'(a,b,d)=\Ind(\Gamma\cap d\SL_3(\Z)_{a,b}d\inv,\Gamma,\Q d[e,b,a]').
\]
  \end{corollary}

We now compute $H_1(\Gamma, C)$.

\begin{theorem}\label{ab}
The group $\Gamma\cap d\SL_3(\Z)_{a,b}d\inv$ is infinite cyclic modulo a subgroup of order $2$.
\end{theorem} 

\begin{proof}
We first work out $\SL_3(\Z)_{a,b}$.  Since 
$b=\trans{(\beta,1,0)}$
and $\beta,1$ are linearly independent over $\Q$, an easy computation shows that 
\[
\SL_3(\Z)_{b} = \{M(h,u)\}
\]
where in $(2,1)$ block form
\[
M(h,u)=\begin{bmatrix}
h&u\\
0&\epsilon
\end{bmatrix}\in\SL_3(\Z),
\]
$u\in\Z^2$, $\epsilon=\pm1$ is chosen to make the determinant equal to $1$, and 
\[
h\begin{bmatrix}\beta\\1\end{bmatrix} = \eta\begin{bmatrix}\beta\\1\end{bmatrix} 
\]
for some $\eta\in E^\times$.  
Because $h\in\SL_2(\Z)$, 
$\eta\in \O_E^\times$ is a unit.  In other words, $h$ is ``$\beta$-unital''.  Note that $h\mapsto\eta$ is one-to-one, so that the $h$'s form a cyclic group modulo $\pm I_2$. 
  
 Now suppose that $M(h,u)$ also stabilizes the $\Q$-line through $a$.  Then $a$ must be an eigenvector of $M(h,u)$ with 
 eigenvalue $\epsilon$.
This happens if and only if
  \[
  h\begin{bmatrix}a_1\\a_2\end{bmatrix} + u = 
  \epsilon\begin{bmatrix}a_1\\a_2\end{bmatrix}. 
  \]
  Therefore $h$, $a$, and  $\epsilon$ determine $u$ uniquely.  It follows that 
  $\SL_3(\Z)_{a,b}$ is infinite cyclic modulo a subgroup of order $2$.
  
  To finish the proof, note that $d\inv \Gamma d\cap \SL_3(\Z)_{a,b}$ has finite index in
  $\SL_3(\Z)_{a,b}$ because $d\inv \Gamma d$ has finite index in $\SL_3(\Z)$.
  \end{proof}
  
  By Shapiro's lemma, $H_1(\Gamma, \I'(a,b,d))$ is isomorphic to 
$H_1( \Gamma\cap d\SL_3(\Z)_{a,b}d\inv,  \Q d[e,b,a]'))$. 
By Theorem~\ref{ab},
 if $\gamma$ is any non-torsion element of the group $\Gamma\cap d\SL_3(\Z)_{a,b}d\inv$, then  
$H_1( \Gamma\cap d\SL_3(\Z)_{a,b}d\inv,   \Q_{a,b})$ is the one-dimensional 
$\Q$-vector space generated by 
\[[\gamma]\otimes_{\Gamma\cap d\SL_3(\Z)_{a,b}d\inv} d[e,b,a]'.\]

\begin{definition}\label{hge}
$H(\Gamma,E)$  is the image of the connecting homomorphism 
\[\psi\colon H_1(\Gamma,C)\to H_0(\Gamma,\st).\]
\end{definition}

\begin{theorem}
Let $\Gamma$ be a subgroup of finite index in $\SL_3(\Z)$, and let $E$ be a real quadratic field.
Then $H(\Gamma,E)$  is the $\Q$-span of all 
 $[f,\gamma f, da]_\Gamma$, where $f,a\in\Z^3$, $d\in\SL_3(\Z)$, and 
 $\gamma\in \Gamma$ arise from all possible choices in the following steps:
\begin{description}
\item[Step 1] Choose a unital $h$, choose $u\in\Z^2$ and set 
$\epsilon=\det(h)$.
\item[Step 2] Form $M(h,u)$, and find an eigenbasis for it of the form
$\set{b,b',a}$, where $b'$ is the Galois conjugate of $b$.  
\item[Step 3] Choose $d\in E(a,b)$.  
\item[Step 4] Find the smallest positive power of $dM(h,u)d\inv$ which lies in $\Gamma$.  
Call this power $\gamma$.  
\item[Step 5] Set $f=de$, where $e$ is the first standard basis vector of $\Q^3$.
\end{description}
\end{theorem}
  
  \begin{proof}
  The notation in the first three steps has already been explained.  For Step 2, note that after we have chosen $h$ and $u$ and formed $M(h,u)$, we see that $M(h,u)$ has an 
eigenvector of the form $b=\trans(\beta,1,0)$, for some $\beta \in E\setminus  \Q$. Since $M(h,u)$ is rational, $b'$ is another eigenvector.  Since $\epsilon = \pm 1$ is another eigenvalue of $M(h,u)$, it has third eigenvector which is rational and which we take to be $a$.

In Step 4, note that 
some power of $dM(h,u)d\inv$ lies in $\Gamma$ because in fact any element of $\SL_3(\Z)$ has some positive power in the finite index subgroup $\Gamma$.

To prove the theorem, take all possible elements of 
$H_1(\Gamma,C)$, and apply the connecting homomorphism 
$\psi\colon H_1(\Gamma,C)\to H_0(\Gamma,\st)$.

 Applying the map induced by inclusion of groups
 \[
 H_1(\Gamma\cap d\SL_3(\Z)_{a,b}d\inv,\Q[de,db,da]')\to H_1(\Gamma,\Q[de,db,da]'),
 \]
  the class of 
  $[\gamma]\otimes_{\Gamma\cap d\SL_3(\Z)_{a,b}d\inv} [de,db,da]'$ maps to the class of 
   $[\gamma]\otimes_\Gamma [de,db,da]'$.
  From Corollary~\ref{c} and the paragraph before the statement of the theorem,  it follows that
   the cycles $[\gamma]\otimes_\Gamma [de,db,da]'$, as $d,a,b,\gamma$ run over all ways of choosing them in Steps 1 through 4, 
   span $H_1(\Gamma,C)$ over $\Q$.  
  Therefore their images under $\psi$ will generate 
  $H(\Gamma,E)$ over $\Q$.
  
  To compute $\psi([\gamma]\otimes_\Gamma [de,da,db]')$
lift this cycle to the chain 
  $[\gamma]\otimes_\Gamma [de,da,db]$, recalling that  $[de,da,db]$ is a modular 
  symbol in $\St(E^3)$.  Note that $\gamma\inv da$ is a multiple of $da$, and $\gamma\inv db$ is a multiple of $db$. Therefore 
\[[\gamma\inv de,\gamma\inv da,\gamma\inv db]=
 [\gamma\inv de,da,db].\]
  
Take the boundary of this chain to obtain
\[
  (\gamma\inv-1)[de,da,db]=[\gamma\inv de, da, db] - [de, da, db]= 
 - [\gamma\inv de,de,da].
\]
The last equality follows by passing $de$ through 
$[\gamma\inv de, da, db]$ and noting that $[\gamma\inv de,de,db]=0$ because 
$\gamma\inv de,de,db$ are not linearly independent over $E$.\footnote{They span a plane in $E^3$ for the following reason.  Note that $\gamma= dM(h,u)^kd\inv$ for some $k>0$.  Therefore $\gamma da=\lambda da$, 
$\gamma db=\mu db$, and $\gamma de = d\theta e$ for some $\lambda,\mu \in E$ 
where we have set $\theta=M(h,u)^k$.  Then $e$,$\theta e$, and $b$ span a plane 
because $b=\trans{(*,*,0)}$.  Therefore the columns of 
 $\gamma[\gamma\inv de,de,db]=  [de, d\theta e, \mu db] = d[e, \theta e, \mu b]$ also span a plane.}

Projecting to the coinvariants, we find that 
 \[-\psi([\gamma]\otimes_\Gamma [de,da,db]') =
 [\gamma\inv de,de,da]_\Gamma= [de,\gamma de,da]_\Gamma,\] since $\gamma da$ is a multiple of $da$.
 If we set $f=de$, we can write the final answer as 
 $[f,\gamma f, da]_\Gamma$.  
  \end{proof}
 
 \begin{remark}
Actually, we need not take $\gamma$ to be the smallest positive power that lies in $\Gamma$.  Any positive power which does so will give a result which is a nonzero rational multiple of what is obtained with the smallest power because of the following lemma.
 \end{remark}
 
\begin{lemma}
 For any $m>0$, $[f,\gamma^m f, da]_\Gamma=m [f,\gamma f, da]_\Gamma$.
 \end{lemma}
 
 \begin{proof}
Note that for some $k>0$, $[f,\gamma f, da]_\Gamma=
 [\gamma f,\gamma^2 f, \gamma da]_\Gamma=
  [\gamma f,\gamma^2 f, dM(h,u)^kd\inv da]_\Gamma$.  This in turn equals  $[\gamma f,\gamma^2 f, da]_\Gamma$ because $a$ is an eigenvector for $M(h,u)$ with eigenvalue $\epsilon=\pm1$, and replacing a column in a modular symbol by a multiple of itself does not change the value of the symbol.
Then by passing $f$ through $[\gamma f,\gamma^2 f, da]_\Gamma$, we find it is equal to 
$[f,\gamma^2 f,da]_\Gamma+[\gamma f, f, da]_\Gamma+
[\gamma f,\gamma^2 f,f]_\Gamma$.  

The last term is 0 because $f,\gamma f$ and $\gamma^2f$ all lie in a plane.  To see this, let $H$ be the plane in $\Q^3$ consisting of vectors whose third coordinate is 0.
Then $M(h,u)$ stabilizes $H$.  So $\gamma=dM(h,u)^kd\inv$ stabilizes $dH$.
But $f\in dH$.  So $f, \gamma f$ and 
$\gamma^2 f$ are all in $dH$.

We conclude that $[f,\gamma^2 f, da]_\Gamma=2[f,\gamma f, da]_\Gamma$.
Similarly, one proves that
\begin{align*}
  [f,\gamma^3 f, da]_\Gamma&=
  [\gamma^2 f,\gamma^3 f, da]_\Gamma+[f,\gamma^2 f, da]_\Gamma\\
  &=
  [f,\gamma f, da]_\Gamma+2[f,\gamma f, da]_\Gamma\\
  &=3[f,\gamma f, da]_\Gamma,\end{align*}
and so on.
\end{proof}

We can simplify the algorithm by avoiding the need to compute the set of double coset representatives $E(a,b)$, as follows:  
 In Step 3, $d$ is a ``random'' double coset representative, so  a priori it can be any element of $\SL_3(\Z)$.  So replace Step 3 with
 \begin{description}
 \item[Step $3'$] Choose $d\in\SL_3(\Z)$.
 \end{description}
 But now we can limit the choice of $d$ as follows.    If we multiply $d$ on the left by an element $x$ of $\Gamma$ then the new $d$ is $xd$, the new $f$ is $xf$, the new 
  $\gamma$ is 
  $x\gamma x\inv$, and we get the output
  $[xf,x\gamma x\inv (xf), xda]_\Gamma = [f,\gamma f, da]_\Gamma$.
   So the output doesn't change if we replace $d$ by $xd$.
  
On the other hand, let $P_3$ be the subgroup of $\SL_3(\Z)$ that stabilizes the plane ${}^t(*,*,0)$.  Let $p\in P_3$ and multiply $d$ on the right by $p=\begin{bmatrix}
A&V\\
0&E
\end{bmatrix}$. 
 Let 
 \[
 pM(h,u)p\inv=\begin{bmatrix}
A&V\\
0&E
\end{bmatrix}
\begin{bmatrix}
h&u\\
0&\epsilon
\end{bmatrix}
\begin{bmatrix}
A&V\\
0&E
\end{bmatrix}\inv=
\begin{bmatrix}
h'&u'\\
0&\epsilon
\end{bmatrix}.
\]
Now $h'=AhA\inv$ is again unital, and $u'$ is still in $\Z^2$.

Let $U$ denote the set of unital matrices.  As $(h,u)$ ranges over 
$U\times \Z^2$, so does  $(h',u')$ range over 
$U\times \Z^2$.
We claim that the output of the algorithm for $(dp,h,u)$ is equal to the output for $(d,h',u')$.   Assuming the claim, 
it follows that 
 without loss of generality we can change Step $3'$ to 
 \begin{description}
\item[Step $3''$] Let $d$ range over a (finite) set of representatives of the double cosets 
 $\Gamma\backslash \SL_3(\Z)/P_3$.  
 \end{description}
This is the form of Step 3 that we use in our computations.

To check the claim:  The output for $(dp,h,u)$ is 
\[\mu_1=[dpe,\gamma dpe,dpa]_\Gamma,\]
where $\gamma$ is the smallest power of $dpM(h,u)p\inv d\inv$ which lies in 
$\Gamma$, and $a$ is the rational eigenvector of $M(h,u)$.
The output for $(d,h',u')$ is 
\[\mu_2=[de,\gamma' de,da']_\Gamma,\]
where $\gamma'$ is the smallest power of $dM(h',u')d\inv$ which lies in 
$\Gamma$, and $a'$ is the rational eigenvector of $M(h',u')$.
The claim follows if we can show that $\mu_1=\mu_2$.

Now $pM(h,u)p\inv=M(h',u')$ so that $a'=pa$ and $\gamma'=\gamma$.  Therefore 
\[\mu_1=[dpe,\gamma dpe,dpa]_\Gamma \quad
\text{and}\quad  \mu_2 = [de,\gamma de,dpa]_\Gamma.
\]
The equality of $\mu_1$ and $\mu_2$ follows from the following lemma applied to $V=dH$, $x=dpe$, $y=de$, and $z=dpa$, where $H$ be the plane spanned by the first two standard basic vectors of $\Q^3$.
(Note that $\gamma$ stabilizes $V$ because $M(h',u')$ stabilizes $H$,
and $\gamma(dpa)=\pm(dpa)$ because $a$ is an eigenvector for $M(h,u)$ with eigenvalues $\epsilon=\pm1$.)

\begin{lemma}
  Let $x$, $y$, and $z$ be nonzero vectors in $\Q^3$.  Let $V \subset \Q^3$ be a plane with $x, y \in V$ and $z \not \in V$.  Let $\gamma \in \Gamma$ such that $\gamma V=V$ and $\gamma z = \pm z$.
Then
\[
[x,\gamma x, z]_\Gamma=[y,\gamma y, z]_\Gamma.
\]
\end{lemma}

\begin{proof}
Passing $y$ through $[x,\gamma x, z]_\Gamma$, we have
$$
 [x,\gamma x, z]_\Gamma = 
  [y,\gamma x, z]_\Gamma+ [x,y, z]_\Gamma+ [x,\gamma x, y]_\Gamma.
$$
Since $x,\gamma x$ and $y$ all lie in the plane $V$, the last symbol on the right is 0.

Passing $\gamma y$ through $[y,\gamma x, z]_\Gamma$, we have
$$
 [y,\gamma x, z]_\Gamma = 
  [\gamma y,\gamma x, z]_\Gamma+ [y,\gamma y, z]_\Gamma+
   [y,\gamma x, \gamma y]_\Gamma.
$$
Since $y$, $\gamma x$, and $\gamma y$ all lie in the plane $V$, the last symbol on the right is $0$.
Putting this together we obtain
\[
 [x,\gamma x, z]_\Gamma =   [\gamma y,\gamma x, z]_\Gamma+ [y,\gamma y, z]_\Gamma+ [x,y, z]_\Gamma.
\]
However, since $\gamma\in\Gamma$ and $\gamma z=\pm z$, we have 
\[[\gamma y,\gamma x,  z]_\Gamma
=[\gamma  y,\gamma x,\gamma z]_\Gamma=[y,  x,z]_\Gamma=-[x,y,z]_\Gamma,\]
  and the desired result follows.
\end{proof}

In summary,  we use the following algorithm to find $H(\Gamma,E)$:
\begin{algorithm}\label{algo}
Let $\Gamma$ be a subgroup of finite index in $\SL_3(\Z)$ and $E$ a real quadratic field.  Let $P_3$ be the subgroup of $\SL_3(\Z)$ that stabilizes the plane ${}^t(*,*,0)$.
Then $H(\Gamma,E)$ is the $\Q$-span of all 
 $[f,\gamma f, da]_\Gamma$ where $f,a\in\Z^3$, $d\in\SL_3(\Z)$, and 
 $\gamma\in \Gamma$ are found by the following algorithm: 
\begin{description}
  \item[Step 1] Choose a unital $h$, choose $u\in\Z^2$, and set 
  $\epsilon=\det(h)$.
 \item[Step 2] Form $M(h,u)$ and find an eigenbasis for it of the form
  $\set{b,b',a}$, where $b'$ is the Galois conjugate of $b$.  
  \item[Step $3''$] Choose $d$ in a set of representatives 
 of the double cosets $\Gamma\backslash \SL_3(\Z)/P_3$.  
 \item[Step 4] Find the smallest positive power of $dM(h,u)d\inv$ 
  which lies in $\Gamma$.  
  Call this power $\gamma$.  
  \item[Step 5] Set $f=de$, where $e$ is the first standard basis vector of $\Q^3$.
  \item[Step 6] Repeat Steps 1--5, going through all possible choices.  
\end{description}
 \end{algorithm}
  
  Of course, in practice we cannot go through an infinite number of choices.  We explain how we deal with this in Section~\ref{comp}.
 
\section{Maximal cusps and their stabilizers}\label{cusps}
Let $n\ge2$, and fix a level $N$.  Let $\Gamma=\Gamma_0(N,n)$, and 
\[\Gamma(N)=\set{g\in\SL_n(\Z) \given g \equiv I \pmod N}.\]

\begin{lemma}\label{lemma1}
Let $u,v\in\Z^n$ be primitive column vectors such 
that $u\equiv v \pmod N$.  Then there exists $g\in\Gamma(N)$ such that $gu=v$.
\end{lemma}

\begin{proof}
First, assume that $u=e_1=\trans{(1,0,\dots,0)}$.
Let $h\in\SL_n(\Z)$ have first column equal to $v$.  So $hu=v$. Let an overline denote reduction modulo $N$.
Then $\overline h\inv \in\SL_n(\Z/N\Z)$ in $(1,n-1)$ block diagonal form looks like
\[
\begin{bmatrix}
1&x\\
0&y
\end{bmatrix}.
\]
Since $y\in\SL_{n-1}(\Z/N\Z)$, there is a matrix $Y\in \SL_{n-1}(\Z)$ such that 
$\overline Y=y$.  Choose $X\in\Z^{n-1}$ such that $\overline X=x$, and set
\[
k=\begin{bmatrix}
1&X\\
0&Y
\end{bmatrix}.
\] 
Then $\overline k = \overline h\inv$ and $k e_1=e_1$.  Hence
$hku=v$ and $\overline{hk}=\overline I$, which implies that $hk\in\Gamma(N)$.

Now let $u$ be general, and choose $A\in\SL_n(\Z)$ such that $Au = e_1$.  Now $Au\equiv Av \pmod N$, so applying what we have already proved to $Au$ and $Av$, we obtain $g'\in\Gamma(N)$ such that $g'Au=Av$.  Take $g=A\inv g' A$.
\end{proof}

Next, we describe the $\Gamma$-orbits of maximal parabolic subgroups of $\GL_3(\Q)$.  We have to correct \cite[Theorem~6]{ash-direct-sum}.
(Although that theorem is wrong in general, it is correct for square-free $N$, so the rest of that paper, which applies only to square-free $N$, remains correct.) 

First, consider the case of parabolic subgroups that are stabilizers of lines.  Each is determined by the line it stabilizes, so it is equivalent to find 
the $\Gamma$-orbits of primitive column vectors $v\in\Z^3$ modulo $\pm1$.

\begin{lemma}\label{lemma2}
The $\Gamma$-orbits of lines in $\Q^3$ are in 1-1 correspondence with the set of positive divisors of $N$.  If $d$ is such a divisor, the corresponding line is generated by $\trans{(1,d,0)}$.
\end{lemma}

\begin{proof}
Given a line in $\Q^3$, let it be generated  by the primitive vector 
$v=\trans{(x,y,z)}$.  A general element of $\Gamma$ looks like
\[
\gamma=
\begin{bmatrix}
a&b&c\\
D&e&f\\
g&h&i
\end{bmatrix}\in \SL_3(\Z)
\]
with $D\equiv g\equiv 0 \pmod N$.  If $(y,z)=(0,0)$, we may
multiply $v$ by some $\gamma$ with $D\ne0$, so without loss of 
generality $(y,z)\ne(0,0)$.  Then
multiplying $v$ by a suitable $\gamma$ with $b=c=D=g=0$, we may replace
$(y,z)$ with $(d,0)$ where $d=\gcd(y,z)$.  So now $v=\trans{(x,d,0)}$ and $x,d$ are relatively prime.

Let $s$ be an integer such that $x+sd$ is prime, which exists by Dirichlet's theorem. (If $x=0$ then $d=\pm1$ so we can take $s$ to be any prime.)  Multiplying $v$ by 
\[
\begin{bmatrix}
1&s&0\\
0&1&0\\
0&0&1
\end{bmatrix}
\]
allows us to assume that $x$ is prime to $N$.

Let $\gamma\in\SL_3(\Z)$ be a matrix that reduces modulo $N$ to
\[
\begin{bmatrix}
\overline x\inv&0&0\\
0&1&0\\
0&0&\overline x
\end{bmatrix}.
\]
Then $\gamma v \equiv \trans{(1,d,0)} \pmod N$.  By Lemma~\ref{lemma1} there exists $g\in\Gamma(N)$ such that $g\gamma v =  \trans{(1,d,0)}$, and $g\gamma\in\Gamma$.

If $\gamma\in\Gamma$ and $\gamma \ \trans{(1,d,0)}=\trans{(1,d',0)}$ for 
some $\gamma\in\Gamma$, then the ideal generated by $d$ in $\Z/N\Z$ equals
 the ideal generated by $d'$ in $\Z/N\Z$.  If $d$ and $d'$ are both positive divisors of $N$,
it follows that
$d=d'$.
\end{proof}

Next, consider the case of parabolic subgroups that are stabilizers of planes.  Let $\tau$ denote the automorphism of $\GL_3$ given by $A\mapsto \trans{A}\inv$.
Then $\tau$ takes these parabolic subgroups to those which are stabilizers of lines.  Let $\Gamma'=\tau(\Gamma)$.  Then the $\Gamma$-orbits of stabilizers of planes are in 1-1 correspondence with the $\Gamma'$-orbits of stabilizers of lines, which are given in the next lemma.

\begin{lemma}\label{lemma3}
The $\Gamma'$-orbits of lines in $\Q^3$ are in 1-1 correspondence with the set of positive divisors of $N$.  If $d$ is such a divisor, the corresponding line is generated by $\trans{(d,1,0)}$.
\end{lemma}

\begin{proof}
Given a line in $\Q^3$, let it be generated by the primitive vector 
$v=\trans{(x,y,z)}$.  A general element of $\Gamma'$ looks like
\[
\gamma=
\begin{bmatrix}
a&b&c\\
D&e&f\\
g&h&i
\end{bmatrix}\in M_3(\Z)
\]
with $b\equiv c \equiv 0 \pmod N$ and of determinant $1$.  If $(y,z)=(0,0)$, we may
multiply $v$ by some $\gamma$ with $D\ne0$, so without loss of 
generality $(y,z)\ne(0,0)$.  Then
multiplying $v$ by a suitable $\gamma$ with $b=c=D=g=0$, we may assume $z=0$.    So now $v=\trans{(x,y,0)}$ and $x,y$ generate the unit ideal in $\Z$.

Let $s\in\Z$ such that $y+sx$ is prime. Multiplying $v$ by 
$$
\begin{bmatrix}
1&0&0\\
s&1&0\\
0&0&1
\end{bmatrix}
$$
allows us to assume that $y$ is prime to $N$.  

If $x=0$, multiplying $v$ by 
$$
\begin{bmatrix}
1&N&0\\
0&1&0\\
0&0&1
\end{bmatrix}
$$
allows us to assume instead that $x\ne0$.  

Write $x=wd$ where $d$ is a positive divisor of $N$ and $w$ is prime to $N$.
Now let $\gamma\in\SL_3(\Z)$ be a matrix that reduces modulo $N$ to
\[
\begin{bmatrix}
\overline w\inv&0&0\\
0&\overline y\inv&0\\
0&0&\overline{wy}
\end{bmatrix}.
\]
Then $\gamma v \equiv \trans{(d,1,0)} \pmod N$.  By Lemma~\ref{lemma1} there exists $g\in\Gamma(N)$ such that $g\gamma v =  \trans{(d,1,0)}$, and $g\gamma\in\Gamma'$.

If $\gamma\in\Gamma'$ and $\gamma \ \trans{(d,1,0)}=\trans{(d',1,0)}$,  then the ideal generated by $d$ in $\Z/N\Z$ is the same as 
the ideal generated by $d'$.
 If $d$ and $d'$ are both positive divisors of $N$,
it follows that
$d=d'$.
\end{proof}

Next, we have to determine $\Gamma\cap P$ where $P$ is a maximal parabolic subgroup.  Consider the exact sequence
\[
1\to U\to P\to P/U\to 1,
\] 
where $U$ is the unipotent radical of $P$ and $P/U$ is isomorphic to any Levi-component $L$ of $P$.  Let $\Gamma_P=\Gamma\cap P$, 
$\Gamma_U=\Gamma\cap U$ and $\Gamma_L=\Gamma_P/\Gamma_U$.  
(Note that it is not necessarily 
true that $\Gamma_L$ is isomorphic to $\Gamma\cap L$.) To compute the homology or cohomology of $\Gamma_P$, we need 
to identify $\Gamma_U$ and  $\Gamma_L$.

First, we deal with stabilizers of lines.  Let $P_0$ denote the stabilizer of the line through $e_1$.  Let $U_0$ be its unipotent radical, and choose a Levi-component $L_0$ as follows:
\[
P_0=
\begin{bmatrix}
*&*&*\\
0&*&*\\
0&*&*
\end{bmatrix},
\quad U_0=
\begin{bmatrix}
1&*&*\\
0&1&0\\
0&0&1
\end{bmatrix},
\quad L_0=
\begin{bmatrix}
*&0&0\\
0&*&*\\
0&*&*
\end{bmatrix}.
\]
Let $\pi\colon P_0\to L_0$ be the obvious projection map, and use it to identify $P_0/U_0$ with $L_0$. We have the exact sequence
\[
1\to U_0\to P_0 \xrightarrow{\pi} P_0/U_0\to 1.
\]

\begin{definition}\label{Gamma*}
Let $M$ and $\Delta$ be positive integers such that $\Delta \mid M$, and set 
\[\Gamma_1(M,\Delta)^*=\set*{g=\begin{bmatrix}a&b\\c&D\end{bmatrix}\in\GL_2(\Z) \given  \text{$c\equiv0 \pmod M$, and $a \equiv \det(g) \pmod \Delta$}}.\]
\end{definition}

\begin{lemma}\label{lines}
Let $d$ be a positive divisor of $N$, and let $\Delta=\gcd(d,N/d)$.
Let $P_d$ be the stabilizer of the line 
through $\trans{(1,d,0)}$.  
Set
\[
g_d=
\begin{bmatrix}
1&0&0\\
d&1&0\\
0&0&1
\end{bmatrix},
\]
and define $\pi_d$ by
\[
\pi_d(\gamma)=\pi(g_d\inv\gamma g_d).
\]
Then we have an exact sequence
\[
1\to \Gamma_{U_d}\to \Gamma_{P_d} \xrightarrow{\pi_d} \Gamma_{L_d}\to 1,
\]
where (i) $\Gamma_{U_d}$ is isomorphic to $\Z^2$ and (ii) $\Gamma_{L_d}$ is isomorphic to $\Gamma_1(N/d,\Delta)^*$.
\end{lemma}

\begin{proof}
Statement (i) is obvious.  
Since $g_de_1=\trans{(1,d,0)}$, $P_0=g_d\inv P_d g_d$ and the definition
 of $\pi_d$ makes sense. Let $\gamma\in M_3(\Z)$, and write
\[
 \gamma=
\begin{bmatrix}
a&b&c\\
D&e&f\\
g&h&i
\end{bmatrix}.
 \]
 Then
\[
g_d\inv \gamma g_d=
\begin{bmatrix}
a+db&b&c\\
-d(a+db)+(D+de)&-db+e&-dc+f\\
g+dh&h&i
\end{bmatrix}.
 \]
 So $\gamma\in \Gamma_{P_d}=P_d\cap\Gamma$ if and only if
 \begin{itemize}
 \item $-d(a+db)+(D+de)=0$;
 \item $g+dh=0$;
 \item $D\equiv g\equiv 0 \pmod N$;
 \item $\det(\gamma)=1$.
 \end{itemize}
 Note that in this case, 
\[
\pi_d(\gamma)=
\begin{bmatrix}
-db+e&-dc+f\\
h&i
\end{bmatrix}
\]
and $ \det (\pi_d(\gamma)) = a+db = \pm1$.
Therefore, if $\gamma\in \Gamma_{P_d}$, then
$
-(a+db)+(D/d)+e=0,
$
and $D/d$ is a multiple of $N/d$,
whence $a\equiv e \equiv   \det (\pi_d(\gamma))   \pmod \Delta$.
Also $h\equiv0\pmod{N/d}$.  We conclude that 
$\pi_d(\gamma) \in \Gamma_1(N/d,\Delta)^*$.

Conversely, given $g'=\begin{bmatrix}a'&b'\\c'&d'\end{bmatrix}\in
\Gamma_1(N/d,\Delta)^*$, 
we will construct $\gamma\in\Gamma_{P_d}$ 
such that $\pi_d(\gamma)=g'$.

First, set $g=-dc'$,$h=c'$,$i=d'$, $c=0$, and $f=b'$.  Then $b'$, $c'$, and $d'$ are correct, the lower left hand corner of $g_d\inv\gamma g_d$ is $0$, and $g \equiv 0 \pmod{N}$.

Next, leaving $a$ and $b$ as variables, set $e=a'+db$ and $D=-de+da+d^2b$.  This ensures that $a'$ is correct and that
the $(2,1)$ coordinate of $g_d\inv\gamma g_d$ is $0$.  

It remains to choose $a$ and $b$ so that $D \equiv 0 \pmod{N}$ and 
$\det (\gamma) = 1$.  First, we need $N$ to divide $D$. Since \[D=-d(a'+db)+da+d^2b)=d(a-a'),\]
 we set $a=a'+X(N/d)$, where $X$ is an unknown integer.

Lastly, we need $\det(\gamma)=1$.  Now $\det(\gamma)=\det( g')(a+db)$.  
Since $\det (g')=\pm1$ we just need to ensure that $a+db=\det(g')$.  Now
\[
a+db=a'+X(N/d) +bd,
\]
and we are given that $a'\equiv \det(g')\pmod \Delta$.  
Write $a'=\det(g') +m\Delta$.  Then we want 
\[
\det (g') +m\Delta+X(N/d) +bd =\det (g'),
\]
so we must choose $X$ and $b$ integers so that 
$m\Delta+X(N/d) +bd=0$.  This can be done because $\Delta=\gcd(d,N/d)$.
\end{proof}

Now we do the same thing for stabilizers of planes. 

\begin{lemma} Let $d$ be a positive divisor of $N$, and let $\Delta=\gcd(d,N/d)$.
Let $P$ be the stabilizer of a plane such that $\trans{P}$ is the stabilizer of the line through $\trans{(d,1,0)}$.
Then we have an exact sequence
\[
1\to \Gamma_{U}\to \Gamma_{P} \to \Gamma_{L}\to 1
\]
where (i) $\Gamma_{U}$ is isomorphic to $\Z^2$ and (ii) $\Gamma_{L}$ is isomorphic to $\trans{\Gamma_1(d,\Delta)}^*$.
\end{lemma}

\begin{proof}
Statement (i) is obvious.
For the rest, as above it suffices to look at 
the stabilizer of a line in $\Gamma'=\trans{\Gamma}$.

Let $d$ be a positive divisor of $N$, and $\Delta=\gcd(d,N/d)$.
Let $P'_d$ be the stabilizer of the line 
through $\trans{(d,1,0)}$.  
Let $P'_0=L'_0U'_0$ be the stabilizer of the line 
through $\trans{(0,1,0)}$ where 
\[
P'_0=
\begin{bmatrix}
*&0&*\\
*&*&*\\
*&0&*
\end{bmatrix},
\ U'_0=
\begin{bmatrix}
1&0&0\\
*&1&*\\
0&0&1
\end{bmatrix},
\ L'_0=
\begin{bmatrix}
*&0&*\\
0&*&0\\
*&0&*
\end{bmatrix}.
\]
Let $\pi'\colon P'_0\to L'_0$ be the obvious projection map and use it to identify $P'_0/U'_0$ with $L'_0$. We have the exact sequence: 
\[
1\to  U'_0 \to  P'_0 \xrightarrow{\pi'}  P'_0/U'_0 \to 1.  
\]
Set
\[
h_d=
\begin{bmatrix}
1&d&0\\
0&1&0\\
0&0&1
\end{bmatrix},
\]
and define $\pi'_d$ by
\[
\pi'_d(\gamma)=\pi(h_d\inv\gamma h_d).
\]
Then we have an exact sequence
\[
1\to \Gamma_{U'_d}\to \Gamma_{P'_d} \xrightarrow{\pi'_d}  \Gamma_{L'_d}\to 1.
\]

 Let $\gamma\in M_3(\Z)$ and write
\[
 \gamma=
\begin{bmatrix}
a&b&c\\
D&e&f\\
g&h&i
\end{bmatrix}.
 \]
 Then
\[
h_d\inv \gamma h_d=
\begin{bmatrix}
a-dD&da+b-d(dD+e)&c-df\\
D&dD+e&f\\
g&dg+h&i
\end{bmatrix}.
\]
 So $\gamma\in \Gamma_{P'_d}=P'_d\cap\Gamma'$ if and only if
 \begin{itemize}
 \item $da+b-d(dD+e)=0$;
 \item $dg+h=0$;
 \item $b\equiv c\equiv 0 \pmod N$;
 \item $\det(\gamma)=1$.
 \end{itemize}
In that case, 
\[
\pi'_d(\gamma)=
\begin{bmatrix}
a-dD&c-df\\
g&i
\end{bmatrix},
\]
and so
\[\det( \pi'_d(\gamma)) = dD+e = \pm1.\]
Therefore, if $\gamma'\in \Gamma'_{P_d}$, then
$
a+b/d-(dD+e)=0
$
and $b/d$ is a multiple of $N/d$,
whence $a\equiv e \equiv   \det(\pi'_d(\gamma))   \pmod \Delta$.
Also $c-df\equiv0\pmod d$.  We conclude that 
$\pi_d(\gamma) \in \trans{\Gamma_1(d,\Delta)}^*$.

Conversely, given $g'=\begin{bmatrix}a'&b'\\c'&d'\end{bmatrix}\in
\trans{\Gamma_1(d,\Delta)}^*$, 
we construct $\gamma\in\Gamma'_{P'_d}$ 
such that $\pi'_d(\gamma)=g'$.

First, set $g=c'$, $h=-dc'$, $i=d'$, $c=0$, and $f=-b'/d$.  Then $b'$, $c'$, and $d'$ are correct, the
$(3,2)$-entry of $h_d\inv\gamma h_d$ is $0$, and $c \equiv 0 \pmod{N}$.

Next, leaving $e$ and $D$ as variables, set $a=a'+dD$ and $b=d^2D+de-da$.  This ensures that $a'$ is correct and that
the $(1,2)$-entry of $h_d\inv\gamma h_d$ is $0$.  

It remains to choose $e$ and $D$ so that $b \equiv 0 \pmod{N}$ and  
$\det(\gamma) = 1$.  First, we need $N$ to divide $b$. Since \[b=d^2D+de-d(a'+dD)=d(e-a'),\]
 we set $e=a'+X(N/d)$, where $X$ is an unknown integer.

Lastly, we need $\det(\gamma)=1$.  Now $\det(\gamma)=\det(g')(dD+e)$.  
Since $\det(g')=\pm1$ we just need to ensure that $dD+e=\det(g')$.  Now
\[
dD+e=dD+a'+X(N/d),
\]
and we are given that $a'\equiv \det(g')\pmod \Delta$.  
Write $a'=\det(g') +m\Delta$.  Then we want 
\[
\det(g') +m\Delta+X(N/d) +dD =\det(g').
\]
So we must choose $X$ and $D$ integers so that 
$m\Delta+X(N/d) +dD=0$.  This can be done because $\Delta=\gcd(d,N/d)$.
\end{proof}

We do not attempt to describe the $\Gamma$-orbits of minimal parabolic subgroups of $\SL_3(\Z)$, but we do need to count them.

\begin{definition}
For any subgroup $G$ of finite index in $\GL_2(\Z)$, let $c(G)$ denote the number of cusps of $H/G$, where $H$ is the upper half plane.
\end{definition}

Note that this is the same as the number of $G$-orbits of lines in $\Q^2$.  We let matrices of negative determinant act on $H$ using the rule that 
$\begin{bmatrix} 1&0\\0&-1 \end{bmatrix} \cdot z = -\overline z$.

\begin{lemma}\label{minimal}
The number of $\Gamma_0(N)$-orbits of minimal parabolic subgroups 
of $\SL_3(\Z)$ is given by the formula
\[
 \sum_{d \mid N, d>0} c(\Gamma_1(d,\Delta)^*), \quad \text{where $\Delta=\gcd(d,N/d)$.}
\]
\end{lemma}

\begin{proof}
Let $\Gamma=\Gamma_0(N)$ as before.
A minimal parabolic subgroup is the stabilizer of a flag, line $\subset$ plane.
We have computed the $\Gamma$-orbits of stabilizers of planes.  They are in 1-1 correspondence with positive divisors $d$ of $N$.  For each such plane, let $P=LU$ be its stabilizer in $\SL_3(\Q)$, so that $\Gamma_P$ is its stabilizer in $\Gamma$.  Then the number of 
$\Gamma_P$-orbits of lines in that plane will equal $c(G)$ where $G$ is quotient of $\Gamma_P$ by $\Gamma_U$.
The lemma now follows from the previous lemmas.
\end{proof}

Let $T$ be the Tits building of $\SL_3(\Q)$.  This is a graph whose vertices are the maximal parabolic subgroups $P$, the edges are the minimal parabolic subgroups $Q$, and $P$ is a vertex of $Q$ if $P\supset Q$.

For any positive integer $M$, let $\tau(M)$ denote the number of positive factors of $M$.

\begin{corollary} The dimension of
$H_1(T/\Gamma_0(N),\Q)$ is given by the formula
\[
b =  \sum_{d \mid N,d>0} c(\Gamma_1(d,\Delta)^*) -  2\tau(N) + 1.
\]
\end{corollary}

To make it easier to compute the number of cusps, note that since $-I_2$ acts trivially on the upper half-plane, the number of cusps of a subgroup $G$ of $\GL_2(\Z)$ will not change if we consider instead the group $\pm G$ generated by $G$ and $-I_2$.

\begin{example}
Let $N=p$ be a prime.  Then $\tau(N) = 2$ and $d=1,p$.  In both cases, $\Delta=1$.  When $d=1$ we have 
\[c(\pm\Gamma_1(1,1)^*)=c(\GL_2(\Z)) = 1,\]
and when $d=p$ we have
\[c(\pm\Gamma_1(p,1)^*)= c(\Gamma_0(p)^\pm) = 2.\]  We obtain \[b=1+2-2\cdot 2+1=0.\]  
\end{example}

\begin{example}\label{psq}
Let $N=p^2$, with $p$ a prime.  Then $\tau(N) = 3$  and  $d=1,p,p^2$.  In the first and third cases cases, $\Delta=1$ and in the middle case $\Delta=p$.  When $d=1$ we have 
$c(\pm\Gamma_1(1,1)^*) = 1$ as before. When $d=p^2$ we have
\[c(\pm\Gamma_1(p^2,1)^*)=c(\Gamma_0(p^2)^\pm) = \frac{p+3}{2}.\]
Finally, when $d=p$, we have
\[c(\Gamma_1(p,p)^*) = (p-1).\]
We obtain 
\[b=1+(p-1)+\frac{p+3}{2}-2\cdot 3+1=\frac{3p-7}{2}.\]
\end{example}

\begin{example}
Let $N$ be a product of $k$ distinct primes 
so that $\tau(N)=2^k$.  Now for each $d$, $\Delta=1$ so
\[
b= \sum_{d\mid N,d>0} c(\Gamma_0(d)^\pm) - 2\cdot 2^k + 1.
\]  

Let $M$ be a positive divisor of $N$ so that $M$ is also square-free.
From \cite[page 38]{shimura}, the number of $\Gamma_0(M)$-orbits of cusps equals
 $\sum_{t \mid M,t>0}\phi(\gcd(t,M/t))=\tau(M)$.  Since the cusps $1/t$ as $t$ varies over the positive divisors of $M$ are pairwise $\Gamma_0(M)$-inequivalent, they give a complete set of representatives of the orbits.  Given any cusp $a/b$, it is in the orbit of $1/t$ where $t=\gcd(M,b)$.

To figure out $c(\Gamma_0(M)^\pm)$, we have to impose a further equivalence under the matrix $J=\diag(-1,1)$.  
Since $J\begin{bmatrix}u&v\\s&t\end{bmatrix}=\begin{bmatrix}-u&-v\\s&t\end{bmatrix}$, $J$ takes the $\Gamma_0(M)$-orbit of the cusp $1/t$ to the $\Gamma_0(M)$-orbit of the cusp $-1/t$, which is the same orbit.  So $c(\Gamma_0(M)^\pm)=
c(\Gamma_0(M))=2^{\tau(M)}$.  

Next, count the factors of $N$ according to how many prime factors each has.  We get
\[ 
\sum_{d\mid N,d>0} {2^{\tau(d)}} =  1 + \binom{k}{1}\cdot 2+ \binom{k}{2}\cdot 2^2 +\cdots+ 2^k = (1+2)^k = 3^k.
\]

Thus
\[
b= 3^k - 2^{k+1} + 1.
\] 
For instance, if $N$ has two distinct prime factors, then $b=9-8+1=2$.
If $N$ has three distinct prime factors, then $b=27-16+1=12$.
\end{example}

We did not attempt to find a 
completely general formula for the number of cusps of $\Gamma_1(d,\Delta)^*$.  Instead, to check our computations, we wrote code to compute 
$c(\Gamma_1(d,\Delta)^*)$ for all positive divisors $d \mid N$, and thus to compute $b$, for all $N\le 50$.  Our other computations were for prime and prime-squared $N$, which are covered by the examples above.  The output confirms our conjecture and acts as a check on our other computations.

\section{Boundary cohomology}\label{bound}

\begin{notation}
For any parabolic subgroup $P$, if we write $P=LU$ we mean that $U$ is the unipotent radical of $P$ and $L$ is a Levi-component of $P$.
\end{notation}

Let $T$ be the Tits building of $\SL_3(\Q)$, which is defined in Section~\ref{cusps}.  Let $\Gamma$ be a congruence subgroup of $\SL_3(\Z)$,
and let $R$ be a ring. 
We recall the structure of the Borel-Serre compactification~\cite{BS}.  

Let $\cS$ be the set of positive-definite symmetric $3\times 3$ matrices modulo homotheties. It is the symmetric space for $\SL_3(\R)$.
Let $X$ be the Borel-Serre compactification of the locally symmetric space 
$\Gamma\backslash \cS$.

The boundary $\partial X$ is the union of the closure of its maximal ($4$-dimensional) faces 
$\overline e'(P)$, where $P$ runs over a set of representatives $\cP$ of $\Gamma$-orbits of maximal parabolic subgroups of $\SL_3(\Q)$. 
 Any two $\overline e'(P)$'s are either disjoint or intersect in a minimal ($3$-dimensional) face. The minimal faces, which are closed, are the nilmanifolds $e'(Q)$, where
$Q$ runs over a set of representatives of $\Gamma$-orbits of minimal parabolic subgroups 
of $\SL_3(\Q)$. 
  The boundary of an $\overline e'(P)$ is a union of those minimal faces $e'(Q)$ such that some minimal parabolic subgroup in the $\Gamma$-orbit of $Q$ is contained in $P$.  
 
 For a maximal parabolic subgroup $P$, let $H^3_!(\overline e'(P),R)$ denote the kernel of the restriction map 
 $H^3(\overline e'(P),R)\to H^3(\partial\overline e'(P),R)$.
 Given a class $z\in H^3_!(\overline e'(P),R)$, it extends by 0 to a class in 
 $H^3(\partial X,R)$ which we  call $z^0$.
 It follows from the Mayer-Vietoris sequence that the span of all such $z^0$ forms a subspace of  $H^3(\partial X,R)$ which restricts isomorphically onto
 $\bigoplus_P  H^3_!(\overline e'(P),R)$.
 
 \begin{definition}
 Let $A'(\Gamma)$ be the $\Q$-span of all $z^0$, where $z\in H^3_!(\overline e'(P),\Q)$ and $P$ runs over $\cP$. 
 \end{definition}

Now assume that 6 is invertible in $R$. 
 Because $X$ is contractible and $\Gamma$ acts properly discontinuously $X$ with finite stabilizers of cardinalities dividing $6^\infty$, there is a natural isomorphism $H^*(\Gamma,R)\simeq H^*(X,R)$ and we will identify these cohomology groups accordingly.

\begin{definition}\label{cuspdef} 
The kernel of the restriction map 
$H^3(X,R)\to H^3(\partial X,R)$ is called the ``interior'' cohomology, and 
denoted by $H_!^3(X,R)$.  
In the introduction, we denoted the image of $H_!^3(X,\Q)$ under the isomorphism 
$H^3(X,\Q)\approx H^3(\Gamma,\Q)$
 by  $H_!^3(\Gamma,\Q)$. 

If $R=\C$, $H^3_\cusp(\Gamma,\C)$ is defined to be the subspace of 
$H^3(\Gamma,\C)$ represented by cuspidal automorphic differential 3-forms on $X$.

If $R$ is a subring of $\C$, $H^3_\cusp(\Gamma,R)$ is defined to be 
$H^3_\cusp(\Gamma,\C)\cap H^3(\Gamma,R)$.

\end{definition}

From \cite{lee-schwermer} we know
\begin{itemize}
\item  If $R$ is a field of characteristic zero, the restriction map 
\[r \colon H^3(\Gamma,R)=H^3(X,R)\to H^3(\partial X,R)\] is surjective.

\item $H_!^3(X,\Q)$, is equal to $H^3_\cusp(\Gamma,\Q)$.

\item  $A'(\Gamma)\otimes\C$ is isomorphic to the direct sum of the spaces of holomorphic cuspforms of weight 2 for 
$\Gamma_L$, where $P=LU$ runs over a set of representatives 
of $\Gamma$-orbits of maximal parabolic subgroups of $\SL_3(\Q)$.

\item $ H^3(\partial X,\Q)= A'(\Gamma)\oplus B'(\Gamma)$, where
$B'(\Gamma)$ is isomorphic to $H_1(T/\Gamma,\Q)$.
\end{itemize}
The decomposition in the last bullet is Hecke-equivariant.  The restriction map in the first bullet is also Hecke-equivariant, and 
we may then conclude: 
 There is a Hecke-equivariant decomposition 
 \[H^3(\Gamma,\Q)=H^3_\cusp(\Gamma,\Q)\oplus A(\Gamma)\oplus B(\Gamma),\]
 where the restriction map $r$ induces isomorphisms $A(\Gamma)\simeq A'(\Gamma)$ and $B(\Gamma)\simeq B'(\Gamma)$.

The dimension of $B(\Gamma)$ equals
the number of $\Gamma$-orbits of minimal parabolic subgroups of $\SL_3(\Q)$ minus the number of $\Gamma$-orbits of maximal parabolic subgroups of $\SL_3(\Q)$ plus 1.

\begin{remark} What Lee and Schwermer actually provide in their paper is the structure of $H^3(\Gamma(N),\C)$ as 
$\SL_3(\Z)/\Gamma(N)$-module, where $\Gamma(N)$ is the principle congruence subgroup of $\SL_3(\Z)$ of level $N$.  We choose $N$ so that $\Gamma(N)$ is contained in $\Gamma$.
Since the $\C$-cohomology is the tensor product of the $\Q$-cohomology
 with $\C$, the bulleted assertions follow by taking 
 $\Gamma/\Gamma(N)$-invariants in $H^3(\Gamma(N),\C)$ and descending to $\Q$.
 \end{remark}
 
 \section{Conjectures}\label{conj}

Our computational results give us confidence to make the conjectures in this section.  We want to state the conjectures in a way that conforms to our method of computation.  For this reason we
have to interpret modular symbols in terms of the Voronoi cellulation.  This gives a concrete realization of the Borel-Serre isomorphism 
$H_0(\Gamma,\St) \approx H^3(\Gamma,\Q)$. For more on the Voronoi cellulation, see Section~\ref{what}.  

\begin{definition} \label{vor-real} 
 Fix as basepoint the identity matrix $I_3$ in the symmetric space $\cS$.  Let $D$ be the closure of the orbit of $I_3$ under the diagonal matrices in the Borel-Serre bordification $\overline\cS$ of $\cS$.  It is a hexagon, in the sense that it is $2$-dimensional and has six edges, each in a different boundary face.
  If $m\in\GL_3(\Q)$ and $[m]_\Gamma$ is the corresponding modular symbol 
  modulo $\Gamma$, define the  
``Voronoi realization'' of $[m]_\Gamma$ to be the projection modulo 
$\Gamma$ of $mD$.  We fix an orientation $o$ on $[I_3]_\Gamma$ and give  $[m]_\Gamma$ the orientation induced from $o$ via the action of $m$ on 
$\overline\cS$.

Define the
``dual Voronoi realization'' of $[m]_\Gamma$ to be $V(m)\in H^3(\Gamma,\Q)$ where $V(m)$ is the Lefschetz dual of the homology class in $H_2(X,\partial X,\Q)$ which is the fundamental class of the Voronoi realization of $[m]_\Gamma$.

If $Y$ is a $\Q$-subspace of $H_0(\Gamma,\St)$, call $V(Y)$ the ``dual Voronoi realization of $Y$'', where
$V(Y)=\set{V(m) \given m\in Y}$.  
\end{definition} 

\begin{lemma}\label{V}
The dual Voronoi realization map $V \colon  H_0(\Gamma,\St) \to H^3(\Gamma,\Q)$ is an isomorphism.
\end{lemma}

\begin{proof}
Unwinding the definitions, it is easy to see that $V$ is injective.  Since the source and target have the same dimension because of Borel-Serre duality, $V$ is also surjective.
\end{proof}

Recall that $A(\Gamma)$ and $B(\Gamma)$ are defined in Section~\ref{bound},  $T$ denotes the Tits building for $\GL_3(\Q)$, and $H(\Gamma,E)$ is given in Definition~\ref{hge}.

\begin{cnj}\label{conj-1}
Let $E$ be a real quadratic field, and let 
$\Gamma\subset\SL_3(\Z)$ be a finite index subgroup.  Then the dual Voronoi realization of $H(\Gamma,E)$ is $H^3_!(\Gamma,\Q)+A(\Gamma)$.
\end{cnj}

\begin{cnj}\label{conj-2}
Let $E$ and $\Gamma$ be as in Conjecture~\ref{conj-1}.  Then
\[
\dim_\Q (H(\Gamma, E))=\dim_\Q (H^3(\Gamma,\Q))-\dim_\Q (H_1(T/\Gamma,\Q)).
\] 
\end{cnj}

\begin{theorem}\label{12}
Conjecture~\ref{conj-1} implies Conjecture~\ref{conj-2}.
\end{theorem}

\begin{proof}
From Section~\ref{bound} we know that 
$H^3(X,\Q)\simeq H^3_!(X,\Q)\oplus A \oplus B$.
 So Lemma~\ref{V} and Conjecture~\ref{conj-1} imply that 
the codimension of  $H(\Gamma,E)$ equals the dimension of $B(\Gamma)$, 
namely $\dim_\Q (H_1(T/\Gamma,\Q))$.
\end{proof}

We first formulated Conjecture~\ref{conj-2} on the basis of purely numerical data from our computations.  Later we formulated Conjecture~\ref{conj-1} and were able to check it by 
computing Hecke operators on $H(\Gamma,E)$ and on $H^3(\Gamma,\Q)$.
See Section~\ref{results}.

Beyond the experimental evidence, we do not know why $V(H(\Gamma,E))$ should contain $H^3_!(\Gamma,\Q)$ nor why its intersection with $B(\Gamma)$ should be trivial.  We have the following heuristic as to why it might contain $A(\Gamma)$, but it is far from being a proof.

Consider the commutative diagram (suppressing the $\Q$-coefficients in the notation):
\[
\xymatrix{
H^3(X) \ar[r] & H^3(\partial X)\\
H_1(X,\partial X) \ar[r] \ar[u]& H_1(\partial X) \ar[u]
}
\]
where the top map is restriction, the bottom map is the boundary map, the left hand vertical map is Lefschetz duality and the right hand vertical map is Poincar\'e duality.
(Since we are using $\Q$-coefficients and the spaces we are considering are orbifolds with finite stabilizers,  Lefschetz and Poincar\'e duality apply to their homology and cohomology.)

If $\cP$ is a set of representatives of $\Gamma$-orbits of maximal parabolic subgroups of $\SL_3(\Q)$,  then $A(\Gamma)=\bigoplus_{P\in\cP}H^3(\overline e'(P),\Q)$. 
By Lefschetz duality $H_!^3(\overline e'(P),\Q)$ is isomorphic to a subspace of 
$H_1(\overline e'(P),\Q)$, call it $H_1^!(\overline e'(P),\Q)$.

Let $m=[f,\gamma f,da]_\Gamma\in H(\Gamma_0(N),E)$, as in Algorithm~\ref{algo}.  Let $P=LU$ be the parabolic subgroup which is the stabilizer of the plane spanned by $f$ and 
$\gamma f$.  Then the Voronoi realization of $m$
in $H_1(X,\partial X)$ is a hexagon,
one edge of which, call it $\eta$, lies in the face $e'(P)$.  The fundamental class of $\eta$ lies in 
$H_1^!(\overline e'(P),\Q)$.

Let $p\colon P\to P/U$ be the projection, $\Gamma_L=p(\Gamma\cap L)$ 
and $\Gamma_U=\Gamma\cap U$.
Then $e'(P)$ is a fibration with base $B_L$ and fiber $F$, 
where $B_L$ is isomorphic to the upper half-plane modulo $\Gamma_L$ and $F$ is the torus 
$U(\R)/\Gamma_U$.  Let $X_L$ denote the Borel-Serre compactification of $B_L$.  Then the projection $\eta'$ to $X_L$ of the edge $\eta$ is the Voronoi realization 
of the modular symbol $[\gamma f,f]_{\Gamma_L}$ in $X_L$.   

From the commutative diagram we can see that the Voronoi dual of $m$ equals the Voronoi dual of $\eta$ (extended by 0) plus five other terms coming from the other edges of the hexagon.   
Now $[\gamma f,f]_{\Gamma_L}$ equals $[\delta f,f]_{\Gamma_L}$, where $\delta$ is a unital element in $\Gamma_L$.
By~\cite[Theorem~9.3]{AY}
(which assumes the Generalized Riemann Hypothesis) the Voronoi duals of  
$[\delta f,f]_{\Gamma_L}$ span $H_!^1(X_L,\Q)$
as we vary $\delta$ over all unital elements of $\Gamma_L$.  

It seems likely that as we vary $\gamma$, $f$, and $a$ we could find a linear combination of 
various $[\gamma f,f]_{\Gamma_L}$'s, whose contribution to
$H_!^1(X_L,\Q)$ equals any desired element, while
the fundamental classes of the other edges of the hexagons cancel out.
This would explain why $A(\Gamma)$ is contained in $V(H(\Gamma,E))$.
  
 \section{Hecke operators and Galois representations}\label{hecke}

The tame Hecke algebra $\HH_{n,N}$ is the commutative $\Z$-algebra under convolution generated by the double cosets $T(\ell,k)=\Gamma_0(n,N) D(\ell,k) \Gamma_0(n,N)$ with 
\[D(\ell,k)=\diag(\underbrace{1,\cdots,1}_{n-k},\underbrace{\ell,\cdots,\ell}_k).\]
for all prime $\ell\nmid pN$.    Let $S_{n,N}$ denote the subgroup of $\GL_n(\Q)$ generated by the elements in all these double cosets.

A Hecke packet over a ring $R$ is an algebra homomorphism $\phi\colon\HH_{n,N}\to R$. If $W$ is an $\HH_{n,N}\otimes R$-module and $w\in W$ is a simultaneous eigenvector for all $T\in\HH_{n,N}$, then the associated eigenvalues give a Hecke packet.

\begin{definition} \label{def:attached} Let $\phi$ be a Hecke packet with $\phi(T(\ell,k))= a(\ell,k)$.
We say that the Galois representation $\rho\colon G_\Q\to\GL_n(R)$ is attached to 
$\phi$ if $\rho$ is unramified outside $pN$ and 
\[\det(I-\rho(\frob_\ell)X)=\sum_{k=0}^n(-1)^k\ell^{k(k-1)/2}a(\ell,k)X^k\]
for all $\ell\nmid pN$.
(This is the arithmetic Frobenius: if $\omega$ be the cyclotomic character, 
$\omega(\frob_\ell)=\ell$.)
When $\phi$ comes from a Hecke eigenvector $w$,
we say that $\rho$ is attached to $w$.
\end{definition}

There is a natural action of a double coset $T(\ell,k)\in\HH_{n,N}$ on the homology $H_*(\Gamma_0(n, N),M)$ and on the cohomology 
$H^*(\Gamma_0(n,N),M)$  for any $S_{n,N}$-module $M$.  Now let $n = 3$ and $\Gamma = \Gamma_0(3,N)$.
 It is known from \cite{HLTT, scholze}  that there is a Galois representation attached to each Hecke eigenclass in $H_\cusp^3(\Gamma,\C)$.   From  \cite[Section~3.2]{ash-stevens}, it follows that
if $z\in A(\Gamma)\otimes_\Q\C$ is a Hecke eigenclass, then it has an attached Galois representation which is the direct sum of a Dirichlet character of conductor dividing $N$ and an odd two-dimensional representation coming from a holomorphic modular form of weight 2 and level dividing $N$ and trivial nebentype. Adapting the proof of this result to the case of a minimal parabolic subgroup, one shows that
if $z\in B(\Gamma)\otimes_\Q\C$ is a Hecke eigenclass, then it has an attached Galois representation which is the direct sum of 3 Dirichlet characters of levels dividing $N$.  
Therefore we can use the attached Galois representations as a quick way of identifying Hecke eigenclasses in $H^3(\Gamma,\C)$.

Because the short exact sequence 
\[
0 \to \St(\Q^3;R) \to \St(E^3;R) \to C\to 0
\]
 is equivariant for the action of $\GL_3(\Q)$, the long exact sequence of homology derived from it is equivariant for $\HH_{3,N}$.  Therefore the connecting homomorphism $\psi$
  is $\HH_{3,N}$-equivariant, and $H(\Gamma,E)$ is $\HH_{3,N}$-stable.  This gives a check on our computations. 
We diagonalize the Hecke operators on $H^3(\Gamma,\Q)$  and $H(\Gamma,E)$, and thereby verify Conjecture~\ref{conj-1} for the range of levels $N$ and fields $E$ specified in the introduction.  See Section~\ref{results} for an example of how we do this.

\section{Voronoi homology} \label{what}

We make the computations using already-existing programs that find the Voronoi homology of arithmetic subgroups of $\GL_n(F)$ for arbitrary number fields $F$.  For this reason, we need to know, when $F=\Q$, that the Voronoi homology over $R$ in degree $0$ is isomorphic to $H_0(\Gamma, \st(\Q^n,R))$.
We also need to make this isomorphism explicit, so that we know how to express the modular symbols in $H(\Gamma,E)$ in terms of the Voronoi homology.  Although in this paper we only need the case $n=3$ and $R=\Q$, for future purposes we take $n\le4$ and work over a more general ring $R$.  We do not know if the analogue of Theorem~\ref{whatwhat} for $n>4$ is true.
For more background on the Voronoi decomposition and Voronoi homology see \cite{voronoi1, AGM2, PerfFormModGrp}. 

First, we describe the Voronoi complex.  To conform with the notation in \cite{AY}, in this section we use a notation that clashes with that of Section~\ref{bound}.
Let $\Gammabar = \GL_n(\Z)$ 
and let $\Gamma$ be a subgroup of finite index in $\Gammabar$.
Let $V$ be the $n(n+1)/2$-dimensional vector space of 
real symmetric 
$n\times n$ matrices. 
 Let $C \subset V$ denote the open cone of
positive definite symmetric matrices, and  $q \colon \Z^n \to V$ be the map
$q(v) = v \ \trans{v}$.  (All vectors in this section are column vectors.)  

For each nonzero, proper subspace $W$ of $\Q^n$, let $C(W)\subset V$ denote the cone of 
positive semi-definite symmetric matrices whose kernel is $W\otimes\R$.  Set $\overline C$ to be the union of $C$ and all the $C(W)$'s.  Let
$X =C/\R_+$ and $\overline X=\overline C/\R_+$,
where $\R_+$ acts on $C$ by scaling.  

The group
$\Gammabar$ acts on $V$: for $\gamma \in \Gammabar$ and $A \in V$,
$\gamma \cdot A = \gamma A \ \trans{\gamma}$.  This action restricts
to an action on $\overline C$, which descends to an action on $\overline X$.  
There is a cellular tessellation of $X$ called the Voronoi decomposition of $\overline X$.   Each cell is the conical convex hull (modulo $\R_+$) of 
$q(v_i), i=1,2,\dots,m$, where $\{\pm v_i\}$ are the minimal vectors of some positive definite real quadratic form.

Assume now that $n\le 4$.  
There is one $\Gammabar$-orbit of cells of dimension $n-1$.  As a representative,
take $\sigma_0$, the one with minimal vectors given by $\set{\pm e_i \given i = 1, 2, \dots, n}$, where 
$\set{e_i}$ is the standard basis of $\Z^n$.
For $g\in\Gammabar$, denote by $\sigma(g)$ the Voronoi cell $g\sigma_0$.  

Let $V^\Z_k$ be the $\Z$-module of oriented $k$-chains in the Voronoi cellulation of 
$\overline X$ modulo $\partial \overline X=\overline X\setminus X$.  It is the 
$\Z$-module generated by all oriented Voronoi cells of dimension $k$ modulo the subspace generated by those which lie wholly in 
$\partial \overline X$.  
If $g\in\Gammabar$ is an orientation reversing element in the stabilizer of a Voronoi cell $\tau$, then $g\tau=-\tau$.  We denote the image of $\sigma(g)$ in $V^\Z_{n-1}$ again by $\sigma(g)$.

Let $\st_\Z=\st(\Q^n;\Z)$ be the Steinberg module with integer coefficients.  Define $\chi\colon V^\Z_{n-1}\to\st_\Z$ by
\[
\chi(\sigma(g))=[g]
\]
for any $g\in\Gammabar$. Then
$\chi$ is equivariant for the action of $\Gammabar$.  Let $\partial\colon V^\Z_n\to V^\Z_{n-1}$ be the boundary map.  This fits into a resolution of $\st_\Z$:

\begin{theorem}[{\cite[Theorem 11]{AGM2}}]\label{agm2}
Let $n\le 4$.  Then
\[
0\to V^\Z_{n(n+1)/2-1}\to\cdots\to V^\Z_n
\xrightarrow{\partial} V^\Z_{n-1} \xrightarrow{\chi} \st_\Z\to0
\]
is an exact sequence of $\Gammabar$-modules
\end{theorem}

This is not a free $\Z[\Gammabar]$-resolution of $\st_\Z$.  For each $k$,  the module $V^\Z_k$ is isomorphic to a direct sum of induced $\Z[\Gammabar]$-modules, each of which is obtained from the orientation character induced from the finite stabilizer of some Voronoi cell to
 $\Gammabar$. Let $d=6$ if $n\le 3$ and $d=30$ if $n=4$. The orders of the stabilizers divide $d^\infty$.  
 
 Set $V_k=V_k^\Z\otimes_\Z R$ for all $k$.
 Then the ``Voronoi homology'' of $\Gamma$ over $R$, which by definition is the homology of the complex of coinvariants,
\[
0\to (V_{n(n+1)/2-1}))_\Gamma\to\cdots\to 
(V_n)_\Gamma \xrightarrow{\bar\partial} (V_{n-1})_\Gamma  \to0,
\]
is isomorphic as an $R$-module to the Steinberg homology 
$H_*(\Gamma,\st(\Q^n,R))$ if $d$ is invertible in $R$.  (See~\cite[Corollary~12]{AGM2}.)
 
 The theorem we need to justify our computations is the following:
\begin{theorem}\label{whatwhat}
Let $n\le 4$.
Let $d=6$ if $n\le 3$ and $d=30$ if $n=4$.
Let $R$ be a ring on which $d$ acts invertibly, and let $\Gamma$ be a subgroup of finite index in $\Gammabar$. 
Then there is an isomorphism 
\[ \phi\colon  \st(\Q^n,R)_\Gamma=H_0(\Gamma,\st(\Q^n,R))\to (V_{n-1})_\Gamma/\image\bar  \partial,\]
where $\phi$ may be computed as follows:  For any $[g]\in\st$, suppose
$[g]=\sum[B_j]$ for some $B_j\in\GL_n(\Z)$.  Then 
\[
\phi([g]_\Gamma)=(\sum \sigma(B_j)_\Gamma)',
\] 
where the subscript $\Gamma$ denotes the image in the coinvariants
and the prime denotes reduction modulo the image of $\bar\partial$.
\end{theorem}

\begin{proof}
In this proof, we write $\st$ as short for $\st(\Q^n,R)=\st_\Z\otimes_\Z R$.
Let $F_\bullet\to R$ be the standard resolution of the trivial module $R$.  So $F_i$ is the free $R[\Gamma]$-module with 
 basis $(g_0,\dots,g_i)\in \Gamma^{i+1}$, and the action is given by
 $g(g_0,\dots,g_i)=(gg_0,\dots,gg_i)$.
 
Then $F_\bullet\otimes_R\st\to \st$ 
is a free resolution of the $R[\Gamma]$-module $\st$.  By the Fundamental Lemma of Homological Algebra (FLHA)~\cite[Chpt.~I, Lemma~7.4]{B}  there is an augmentation preserving chain map $f \colon F_\bullet\otimes_R\st\to V_\bullet$, which is unique up to homotopy.  (Note that $f$ shifts subscripts, taking $F_i\otimes_R\st$ to $V_{i+n-1}$.)

Given any resolution $\Phi_\bullet$ (not necessarily free) of $\st$ by $R\Gamma$-modules, there is a spectral sequence
\[
E_{pq}^1 = H_q(\Gamma, \Phi_p) \Rightarrow H_{p+q} (\Gamma, \Phi_\bullet).
\]
See~\cite[Chpt.~VII, (5.3)]{B}.  
Suppose that for each $p$, the module $\Phi_p$ satisfies the condition:
\begin{itemize}
\item[($\ast$)] $\Phi_p$ is a direct sum of induced modules, each induced from a finite subgroup whose cardinality is invertible in $R$.
\end{itemize}
Then using Shapiro's lemma, we see that $E_{pq}^1=0$ whenever $q>0$.
(Compare~\cite[Chpt.~VII, (7.10)]{B}.)  It follows that 
$H_{\bullet}(\Gamma, \st)$ is isomorphic to the 
homology of the complex $(\Phi_\bullet)_\Gamma$ of
 $\Gamma$-coinvariants.

Note that ($\ast$) holds for both $\Phi=F_\bullet\otimes_R\st$ and $\Phi=V_\bullet$.
Therefore $H_\bullet(\Gamma, \st)$, which equals the homology of the complex $F_\bullet\otimes_\Gamma \st = (F_\bullet\otimes_R \st)_\Gamma$, is isomorphic to the homology of the complex $(V_\bullet)_\Gamma$.  This isomorphism can be induced by the map $f$ on the spectral sequences.
It follows that 
$f$ induces an isomorphism on homology
$H_i(\Gamma, \st)\to H_{i+n-1}((V_\bullet)_\Gamma)$.  This was proved already in~\cite{{AGM2}}.

We now  construct an explicit $f$, and $\phi$ will be $f$ in degree 0.  Following the proof of \cite[Chpt.~I, Lemma~7.4]{B}, we construct $f$ one degree at a time, starting with the identity map on the augmentation module $\st$.  Since we only need to get to $\phi$, we just have to do the first step.

Choose the $R$-basis for $\st$ consisting of $\{[U]\}$ where $U$ runs through the unipotent upper triangular matrices in $\SL_n(\Q)$.  For each $[U]$, there exists a finite set $\{A_i(U)\}\subset\SL_n(\Z)$ such that 
 $[U]=\sum [A_i(U)]$.

  Then $F_0\otimes_R\st$ has a free $R\Gamma$ basis consisting of 
$\{I\otimes [U]\}$, where $I$ denotes the identity matrix in $\Gamma$.  
Define $\Psi(I\otimes [U])=\sum \sigma(A_i(U))\in V_{n-1}$.   Extend $\Psi$ to all of 
$F_0\otimes_R\st$ to make it $\Gamma$-equivariant.  Then $\Psi$ commutes with the augmentation maps and we can take $f_0=\Psi$.

Let $\phi$ be the map on homology induced by $\Psi$.  As shown above, it is an isomorphism.  It remains to show that it can be computed as stated in the theorem.  Let $g\in\GL_n(\Z)$.  Write $[g]=\sum c_U[U]$.  Then 
\[
\phi([g]_\Gamma)=\sum\sum c_U\sigma(A_i(U))_\Gamma.
\]

On the other hand, suppose $[g]=\sum[B_j]$ for some $B_j\in\GL_n(\Z)$.  Then 
$\sum[B_j]-\sum\sum c_U[A_i(U)]=0$ in $\st$.  Therefore by Theorem~\ref{by}, and using its notation, extended to write $\gen{A}$ instead of $\gen{a_1,a_2,\dots,a_n}$ if $A$ is a matrix with the columns $a_1,a_2,\dots,a_n$,
 there is a finite set of $n$-tuples $S$ such that 
 \begin{multline*}
\sum\gen{B_j}-\sum\sum c_U\gen{A_i(U)} =
\\
\sum_{(a,b,a_3,\dots,a_n)\in S}
\gen{a,b,a_3,\dots,a_n}+\gen{-b,a+b,a_3,\dots,a_n}+\gen{a+b,-a,a_3,\dots,a_n},
 \end{multline*}
where  each 
  $\set{a,b,a_3,\dots,a_n}$ is a $\Z$-basis of $\Z^n$.
  It follows that 
  \begin{multline*}
\sum\sigma(B_j)-\sum\sum c_U\sigma(A_i(U)) =\\
\sum_{{a,b,a_3,\dots,a_n}\in S}
\sigma(a,b,a_3,\dots,a_n)+\sigma(-b,a+b,a_3,\dots,a_n)+\sigma(a+b,-a,a_3,\dots,a_n).
  \end{multline*}

There is a Voronoi $n$-cell $\tau_0$ whose boundary is
  $$\sigma(e_1,e_2,e_3,\dots,e_n)+\sigma(-e_2,e_1+e_2,e_3,\dots,e_n)
  +\sigma(e_1+e_2,-e_1,e_3,\dots,e_n).$$ 
  ($\tau_0$ has minimal vectors $\pm\{e_1,\dots, e_n, e_1+e_2 \}$). 
  Therefore
    \[\Sigma\coloneqq \sigma(e_1,e_2,e_3,\dots,e_n)+\sigma(-e_2,e_1+e_2,e_3,\dots,e_n)
  +\sigma(e_1+e_2,-e_1,e_3,\dots,e_n)\]
  is in the image of $\partial$.  If $h$ is the matrix with columns $a,b,a_3,\dots,a_n$, then \[h\Sigma=\sigma(a,b,a_3,\dots,a_n)+\sigma(-b,a+b,a_3,\dots,a_n)+\sigma(a+b,-a,a_3,\dots,a_n)\] is in the image of $\partial$.
  
Therefore 
    \[
\sum\sigma(B_j)-\sum\sum c_U\sigma(A_i(U)) 
\]
   is in the image of $\partial$, and its image in the coinvariants  is in the image of $\bar\partial$.
  
We are trying to prove   $\phi([g]_\Gamma)= \sum \sigma(B_j)_\Gamma$ modulo the
   image of $\bar\partial$.
We have seen that $\phi([g]_\Gamma)=\sum\sum c_U\sigma(A_i(U))_\Gamma$.
  So we are finished by the preceding paragraph.
    \end{proof}

\section{Description of the computations} \label{comp} 
We now give the details of the computation for $n = 3$.  Let $v_1, v_2, \dots, v_6$ be the column vectors of the matrix $\mat{ 1& 0& 0& 1& 0& 1\\ 0& 1& 0& 1& 1& 1\\ 0& 0& 1& 0& 1& 1 }$.  For a subset $S \subseteq \set{1, 2, \dots, 6}$, we abuse notation and let $S$ also denote the conical convex hull (modulo $\R_+$) of $\set{q(v_{a})\given a \in S}$.  Then the Voronoi cellulation of $\Xbar$ comes from the $\Gammabar$-orbit of the simplicial  $5$-cell $\set{1, 2, \dots, 6}$ and its faces.

Fix $\Gamma = \Gamma_0(N)$.  We need to compute 
$(V_3)_\Gamma$ and $(V_2)_\Gamma$.
There are two $\Gammabar$-orbits of $3$-cells, and we fix representatives $\sigma_1 = \set{1, 2, 3, 4}$ and $\sigma_2 = \set{1, 3, 4, 5}$.  There is one $\Gammabar$-orbit of $2$-cells, with representative $\sigma_0 = \set{1,2,3}$ as described in Section~\ref{what}.  We follow \cite[Section~3]{AGM1} 
to compute $(V_{3})_\Gamma$ and $(V_{2})_\Gamma$.  
The $i$-cells in $(V_i)_\Gamma$ are in one-to-one correspondence with the orientable right $\Gammabar_T$-orbits in $\PP^2(\Z/N\Z)$, where $T = \sigma_0$ for $i = 2$ and $T = \sigma_1, \sigma_2$ for $i = 3$, and $\Gammabar_T$ is the stabilizer of $T$.  

The boundary differential $\bar \partial$ is computed as described in 
\cite[Section~6.2]{imquadcoh}.  After one has the differential, computing the homology is a standard problem in exact linear algebra, since the Voronoi homology in bottom degree is the cokernel $\Coker(\bar \partial)$.

We also need to compute the Hecke operators $T(\ell,k)$ for primes $\ell \nmid N$ and $k = 0, 1, \dots, 3$.  Because the coefficient module is trivial, we have $T(\ell,0)$ and $T(\ell,3)$ are each the identity map.  As described in Section~\ref{hecke}, set
\[D(\ell,1)=\diag(1,1,\ell) \quad \text{and} \quad D(\ell,2)=\diag(1,\ell,\ell).\]
For $k = 1,2$ the action of the Hecke operators comes from expressing the double coset as a disjoint union of left cosets 
\[\Gamma D(\ell,k) \Gamma = \coprod_{h \in \Omega_k} \Gamma h.\]
We can take
\[
\Omega_1 = \set*{\mat{\ell \\ & 1 \\ && 1}} \cup \set*{ \mat{1 &&
    a\\ & 1 & b\\ && \ell}  \given a, b \in R} \cup
\set*{\mat{1 & a &\\ & \ell & 
  \\ && 1} \given a \in R }.
\]
and 
\[
\Omega_2 = \set*{\mat{\ell \\ & \ell \\ && 1}} \cup \set*{ \mat{1 &
    a& b
    \\ & \ell & \\ && \ell}  \given a, b \in R} \cup
\set*{\mat{\ell &  &\\ & 1 & a
  \\ && \ell} \given a \in R }.
\]
Then the action of $T(\ell,k)$ on the symbol $[v_1, v_2, v_3]$ is given by 
\[T(\ell,k) [v_1, v_2, v_3] = \sum_{h \in \Omega_k}[hv_1, hv_2, hv_3].\]
See~\cite[Section~6 (B)]{AGG} for this formula.  Note that the matrices $h$ have determinant $\ell^k$, so as $\ell$ increases, the determinants of the modular symbols in this formula increase quickly, and so does the number of terms in the sum.  We use the modular symbol algorithm to reduce these modular symbols to ``unimodular" symbols, i.e., those with determinant $1$.  The unimodular symbols represent cohomology classes that are immediately interpretable in terms  of the Voronoi homology, as in Theorem~\ref{whatwhat}.

We now briefly describe the modular symbol algorithm.  This is used both to compute $H(\Gamma,E)$ itself, and the action of Hecke operators on it. We just explained why we need it for the Hecke operators.  As for $H(\Gamma,E)$, we compute elements of it using Algorithm~\ref{algo}.  What emerges from that algorithm are modular symbols $[f,\gamma f,da]_\Gamma$, which generally have large determinant.  They must be reduced to unimodular symbols before we can use 
Theorem~\ref{whatwhat} to interpret them in
the Voronoi homology in degree 0.
The modular symbol algorithm expresses arbitrary modular symbols as sums of unimodular symbols.

It suffices to be able to solve the following. Given an integral matrix $A= \mat{v_1 &  v_2 &  v_3} \in \GL_3(\Q) \setminus \GL_3(\Z)$, produce a nonzero vector $x \in \Z^3$ such that $\abs{\det(A_i)} < \abs{\det(A)}$ for $i = 1, 2, 3$, where $A_i$ is the $3 \times 3$ matrix obtained by replacing the $i$th column by $x$.  Then passing through $x$ on the symbol $[v_1, v_2, v_3]$ as described in Theorem~\ref{thm:mod-props} and repeated application gives the desired modular symbol algorithm.

To find $x$, we use \cite[Section~2.10]{vanGeemenetal} which uses parts of \cite{ash-rudolph}.  We give the main idea here for the convenience of the reader.  Since $\abs{\det(A)} > 1$, there exists an integer $m > 1$ such that $\det(A) \equiv 0 \pmod{m}$.  Thus the nullspace of $A$ modulo $m$ is nontrivial.  It follows that there exists $a_i \in \Z$,  $i = 1, 2, 3$ such that $x = \frac{1}{m}(a_1v_1 + a_2v_2 + a_3 v_3)$ has integral entries.  Since the $a_i$ are representatives of integers modulo $m$, we can choose $a_i$ such that $\abs{a_i} \leq m/2$.  Then
\[\abs{\det(A_i)} = \frac{\abs{a_i}}{m}\abs{\det(A)} \leq \frac{1}{2}\abs{\det(A)},\]
as desired.

We carried out the computations using Magma V2.25-8 \cite{magma} running on Intel(R) Xeon(R) Gold 6148 CPU  2.40GHz.  The server had 16 processors and 64Gb RAM with 100Gb swap. Our computations were not memory intensive, so the swap was not utilized.

For the range of our computations, over various levels $N\le169$ and real quadratic fields $E=\Q(\sqrt\Delta)$ with squarefree $\Delta\le10$, computing the Voronoi homology is not the bottleneck.  As described above, orbits of the projective plane over $\Z/N\Z$ are used to parametrize cells, so the cardinality $\#\PP^2(\Z/N\Z)$ gives a sense of the complexity.  In the range of computation, this was largest for $N = 13^2 = 169$, and in that case, $\#\PP^2(\Z/N\Z) = \num{30927}$.  This results in a boundary matrix of size $1266 \times 3768$.  The orbit computation, the relations, and resulting Voronoi homology took around 11 seconds total in this case.  The computation of the image of $H(\Gamma,E)$ took significantly longer.  This is because the determinants of the outputted modular symbols can become very large.

Now we describe how we used Algorithm~\ref{algo} to compute elements of 
$H(\Gamma,E)$.  Since $H(\Gamma,E)$ is the image of the connecting homomorphism, we call the $\Q$-span of the elements constructed at any stage of the computation the ``image".

With the exception of $N = 121$ and $N = 169$, we conducted an exhaustive search over a rectangular box:

In Step 1 of Algorithm~\ref{algo}, we first construct a unital $2 \times 2$ matrix $h$ as described in \cite[Theorem~4.1.]{AY}.  For the convenience of the reader, we give additional details here.
The task is to generate $\beta$-unital matrices for $\beta \in E\setminus \Q$.  A matrix $h$ is $\beta$-unital if and only if there is a unit $\eta \in \O_E^\times$ such that 
\[h \mat{\beta\\1} = \eta \mat{\beta\\1}.\]
Rather than find $h$ for each $\beta$, we instead find $h$ that admit such $\beta$ by running over the units $\eta = \pm \epsilon^k$, $k=1,2,3,\dots$.   Let $f(x) = x^2 - tx + n \in \Z[x]$ be the minimal polynomial of $\eta$, so $t$ is the trace of $\eta$ and $n \in \set{1, -1}$ is the norm of $\eta$.  Then we take $h \in \Gamma$ of the form
\[h = \mat{t - d & -\frac{f(d)}{cN}\\cN & d}\]
for some integers $c$ and $d$, since the minimal polynomial of $h$ is equal to the minimal polynomial of $\eta$.  Thus we find $h$ by finding integers $c, d$ such that $f(d) \equiv 0 \bmod{N}$ and $c$ is a divisor of $f(d)/N$. 

We proceed as follows.  For each unit $\eta = \pm \epsilon^k$, we find $d'\in \Z/N\Z$ such that $f(d') \equiv 0 \pmod N$.  For each such $d'$, we find $10$ values of $d\in \Z$ in that class modulo $N$.  For each such $d$, we compute $f(d)/N$.  Each factor (positive and negative) of $f(d)/N$ yields a value of $c\in \Z$, and each pair $d$, $c$ gives rise to a $\beta$ and an $h$ that is $\beta$-unital.  In practice, $f(d)$ becomes large quickly, so factoring $f(d)/N$ to find values of $c$ is computationally expensive.  Because of this, if $f(d)$ is too large, we do not try to get a full factorization, instead relying on a partial factorization to get some factors $c$ before just moving on.

The box is determined by taking all powers $-15 \leq k \leq 15$ and all $u \in \Z^2$ of the form $\trans(a, b)$ for $a, b \in \set{-1, 0, 1}$ in Step 1, and taking a complete set of representatives $d$ in Step $3''$.

For each pair $N,E$, the image stabilized within a few minutes, but it took up to several hours to complete the exhaustive search.  
For $N = 121$ and $N = 169$ and certain $E$, a different approach was required because it was taking too long to search the whole box exhaustively.  This was owing to the large number of huge determinant modular symbols that needed to be reduced.  
The modular symbol algorithm, which has to be applied repeatedly, became a time bottleneck.

For these two levels, we switched to another approach.  We selected random elements in the box instead of going through the box systematically.  We then applied Hecke operators to construct more elements of the image.  Thus, for these two levels, we used the known theoretical fact that $H(\Gamma,E)$ is stable under the Hecke algebra. In these two cases, we lost the ability to check our computations by verifying that the image is Hecke stable, since the method ensured Hecke stability.  We still had the check that the image coincided with the prediction of our conjecture.  With this approach for $N = 121$ and $N = 169$, the image stabilized in a few minutes.  

\section{Results of computations}\label{results}

\subsection{Computed dimensions}\label{range}
We compared our computed third betti number for $\Gamma_0(N)$ with Table~1 in \cite{vanGeemenetal}.  Marc Masdeu communicated to us an unpublished table he produced in 2016 with Jun Bo Lau that gave the dimension of $H^3(\Gamma_0(N),\Q)$ for $N\le371$.
Their results disagreed with the table in \cite{vanGeemenetal} for $N=27$, $92$, and $93$.  The Bo--Masdeu numbers are $5$, $41$, and $24$, respectively as opposed to the entries of $6$, $24$, and $41$ given in \cite{vanGeemenetal}.  Our results for $N= 27$, $92$, and $93$ agree with Bo--Masdeu.  The $6$ in \cite{vanGeemenetal} was no doubt a misprint, and the betti numbers for $92$ and $93$ were presumably accidentally interchanged.

We then computed the dimension of $H(\Gamma_0(N),E)$ and verified Conjecture~\ref{conj-2} for $N\le 50$, prime $N<100$, and $N=11^2,13^2$, for all $E=\Q(\sqrt \Delta)$ with  
$\Delta=2,3,5,6,7,10$.

\subsection{Cuspidal cohomology} Using the Hecke analysis, explained in the next subsection, we could verify the constituents of $H(\Gamma,E)$, not only its dimension.
Perhaps our most surprising finding, before we formulated our conjecture, was that the Voronoi dual of $H(\Gamma,E)$ always included the entire cuspidal cohomology.  For the levels $N$ treated in this paper, $H_\cusp^3(\Gamma_0(N),\C)$ is nonzero for $N=53$, $61$, $79$, $89$, $121$ (where it has dimension $2$ for each level), and $169$ (where it has dimension $4$).   This seems to be enough data to justify the part of Conjecture~\ref{conj-1} dealing with the cuspidal cohomology.

\subsection{Hecke analysis and verification of the data}
Fix $N$ and a real quadratic field $E$ in the range described in subsection~\ref{range}.  Let $\Gamma = \Gamma_0(N)$.  
For $\ell \nmid N$, we computed the Hecke operators $T(\ell,k)$ for prime $\ell < 100$ and $k=0,1,2,3$
on $H^3(\Gamma,\Q)$ and on $H(\Gamma,E)$ using modular symbols as described in Section~\ref{what}.  In every case, 
$H(\Gamma,E)$ was found to be Hecke-stable, as it must be.
\footnote{As explained in the previous section, when $N=121$ or $169$, the Hecke stability was actually produced by our computations, rather than ``found to be" so.}  Here is more detail on the Hecke analysis of our data:

First, we simultaneously diagonalized the Hecke operators $\set{T(\ell,k)\given k = 1, 2; \ell < 100}$ acting on $H^3(\Gamma,\Q)$.  
Then for each
Hecke eigenclass, with Hecke eigenvalues $a(\ell,k)$ 
\footnote{Because the coefficient module is trivial, we automatically have that 
 $a(\ell,0)=a(\ell,3) = 1$. }
we formed the Hecke polynomials
\[P_{\ell}(X) =  1-a(\ell,1)X + a(\ell,2)\ell X^2 - \ell^3 X^3,\]
and considered compatible families of Galois representations $\rho \colon G_\Q\to \GL_3(K)$ (for suitable $p$-adic fields $K$) attached to this class as described in Definition~\ref{def:attached}, i.e., \[P_{\ell}(X)=\det(1-\rho(\frob_\ell)X).\]

We can determine the constituents of the decomposition 
\[H^3(\Gamma,\Q)=H_\cusp^3(\Gamma,\Q)\oplus A(\Gamma) \oplus B(\Gamma)\] from
Section~\ref{bound}
by means of the Hecke packets, or equivalently in terms of the attached Galois representations.  This is explained in Section~\ref{hecke}. Namely,
if $z\in A(\Gamma)\otimes_\Q\C$ is a Hecke eigenclass, then its attached Galois representation is the direct sum of a Dirichlet character of conductor dividing $N$ and an odd two-dimensional representation coming from a holomorphic modular form of weight 2 and level dividing $N$ and trivial nebentype.  If $z\in B(\Gamma)\otimes_\Q\C$ is a Hecke eigenclass, its attached Galois representation is the direct sum of 3 Dirichlet characters of levels dividing $N$.  
Any other Hecke eigenclasses belong to $H_\cusp^3(\Gamma,\C)$.
\footnote{Because the Hecke eigenvalues need not lie in $\Q$, individual Hecke eigenclasses will be found in $H^3(\Gamma,F)$ for some (possibly nontrivial) extension $F/\Q$.   However, if a Hecke eigenclass belongs to $A(\Gamma)\otimes \C$, for example, all its Galois conjugates will also lie there, so that $A(\Gamma)$ is defined over $\Q$.  Similarly the other constituents are also defined over $\Q$, and the decomposition occurs over $\Q$, as implied by the formula 
$H^3(\Gamma,\Q)=H_\cusp^3(\Gamma,\Q)\oplus A(\Gamma) \oplus B(\Gamma)$.}

In this way, we compared the Hecke polynomials and the attached Galois representations of the eigenclasses in our computed 
$H(\Gamma,E)$ with our predictions, and we verified Conjecture~\ref{conj-1} in all cases.  

\subsection{\texorpdfstring{$N = 121$}{N=121} example}
Let $\Gamma = \Gamma_0(121)$.  The cohomology $H^3(\Gamma,\Q)$ is $29$-dimensional.  For each real quadratic field $E$ in the range of computation, $H(\Gamma,E)$ is $16$-dimensional.  The cohomology $H^3(\Gamma,\Q)$ decomposes as follows.

\subsubsection{Trivial: \texorpdfstring{$1$}{1}-dimensional contribution}\label{triv}
 There is a $1$-dimensional space where $a(\ell,1) = a(\ell,2) = \ell^2 + \ell + 1$, so that 
  \[P_{\ell}(X) = (1 - X)(1 - \ell X)(1 - \ell^2 X).\]
  This class contributes to $B(\Gamma_0(121))$.
  
\subsubsection{\texorpdfstring{$\GL_2$}{GL2}-newforms: \texorpdfstring{$8$}{1}-dimensional contribution} \label{sssec:gl2} The space of weight $2$ holomorphic cuspforms with trivial character at level $121$ has a $4$-dimensional newspace.  The four newforms have LMFDB \cite{lmfdb} labels \href{https://www.lmfdb.org/ModularForm/GL2/Q/holomorphic/121/2/a/a/}{121.2.a.a}, 
\href{https://www.lmfdb.org/ModularForm/GL2/Q/holomorphic/121/2/a/b/}{121.2.a.b}, \href{https://www.lmfdb.org/ModularForm/GL2/Q/holomorphic/121/2/a/c/}{121.2.a.c}, and \href{https://www.lmfdb.org/ModularForm/GL2/Q/holomorphic/121/2/a/d/}{121.2.a.d}.  
The Hecke field of each new newform is $\Q$.

Each newform $f$ contributes a $2$-dimensional space to $H^3(\Gamma,\Q)$ in the following way.  Let $a_\ell = a_\ell(f)$.  There are two Hecke eigenforms in $H^3(\Gamma,\Q)$ with rational Hecke eigenvalues $\set{a(\ell,k)}$ and $\set{b(\ell,k)}$ such that 
  \[a(\ell,1) = b(\ell,2) = a_\ell + \ell^2 \quad \text{and} \quad a(\ell,2) = b(\ell,1) = \ell a_\ell + 1. \]  Then the associated Hecke polynomials are 
  \[P_{\ell}(X) = (1 - \ell^2X)(1-a_\ell X + \ell X^2)\]
  and
\[P_{\ell}(X) = (1 - X)(1 -\ell a_\ell X+ \ell^3 X^2).\]
\subsubsection{\texorpdfstring{$\GL_2$}{GL2}-oldforms: \texorpdfstring{$6$}{6}-dimensional contribution}  The space of weight $2$ holomorphic cusp forms with trivial character at level $11$ is $1$-dimensional, and it is new since $11$ is prime. It is generated by the class with LMFDB label \href{https://www.lmfdb.org/ModularForm/GL2/Q/holomorphic/11/2/a/a/}{11.2.a.a}. This $\GL_2$ cuspform $f$ contributes to $H^3(\Gamma,\Q)$ with eigenforms as in the $\GL_2$-newform case, but now each with multiplicity $3$.  This is consistent with the observation in \cite[Remark~3.2]{vanGeemenetal} who cite \cite{reeder}. 

These classes from the newforms and the oldforms of level $121$ together span $A(\Gamma_0(121))$, which thus has dimension 14.

\subsubsection{Symmetric squares: \texorpdfstring{$2$}{2}-dimensional contribution}
The $\GL_2$-newform of level $121$ with LMFDB label \href{https://www.lmfdb.org/ModularForm/GL2/Q/holomorphic/121/2/a/d/}{121.2.a.d} is a quadratic twist of the newform of level $11$ with LMFDB label  \href{https://www.lmfdb.org/ModularForm/GL2/Q/holomorphic/11/2/a/a/}{11.2.a.a}. 
  Then the symmetric square lift of these $\GL_2$-forms agree and give a cuspidal Hecke eigenclass in $H^3(\Gamma,\Q)$
with Hecke eigenvalues $\set{a(\ell,k)}$ such that 
\[a(\ell,1) = a(\ell,2) = a_\ell^2 - \ell, \]
where $a_\ell$ is the Hecke eigenvalue of the $\GL_2$-form.
The associated Hecke polynomials are 
  \[P_{\ell}(X) = (1 - \ell X)(1-(a_\ell^2 - 2 \ell) X + \ell^2 X^2).\]
  
  The two $\GL_2$-newforms of level $121$ with LMFDB labels \href{https://www.lmfdb.org/ModularForm/GL2/Q/holomorphic/121/2/a/a/}{121.2.a.a}
  and  \href{https://www.lmfdb.org/ModularForm/GL2/Q/holomorphic/121/2/a/c/}{121.2.a.c}
   are quadratic twists of each other, so they have the same symmetric square lift in $H^3_\cusp(\Gamma,\Q)$ constructed as above.

The remaining $\GL_2$-newform at level $N = 121$ is  
\href{https://www.lmfdb.org/ModularForm/GL2/Q/holomorphic/121/2/a/b/}{121.2.a.b}. 
It has CM and its symmetric square lift (which is not cuspidal) has level higher than $121$ and does not appear in $H^3(\Gamma,\Q)$.

In sum: $H^3_\cusp(\Gamma,\Q)$ is made up of symmetric square lifts, and has dimension 2.

\subsubsection{Dirichlet characters: \texorpdfstring{$12$}{12}-dimensional contribution}
The group of Dirichlet characters of modulus $11$ is a cyclic group of order $10$.  There is an index $2$ subgroup of order $5$.  The four nontrivial elements of this subgroup each contributes a $3$-dimensional subspace of $H^3(\Gamma,\Q)$ as follows.
Let $\chi$ be one of the characters of order $5$.  Then $\chi$ has codomain the cyclotomic field $\Q(\zeta_5)$.  The character 
$\chi$ contributes $3$ Hecke polynomials:
\begin{align*}
P_\ell(X)& = (1-X)(1-\chi(\ell)\ell X)(1-\chi^{-1}(\ell)\ell^2 X) \\
P_\ell(X)& = (1-\ell X)(1-\chi(\ell) X)(1-\chi^{-1}(\ell)\ell^2 X) \\
P_\ell(X)& = (1-\ell^2 X)(1-\chi(\ell) X)(1-\chi^{-1}(\ell) \ell X).
\end{align*}

These classes are all in $B(\Gamma_0(121))$ and together with the class in subsection~\ref{triv} give $\dim_\Q(B(\Gamma_0(121)) =13$.  Note that $13$ equals $b=(3\cdot 11-7)/2$ as predicted in 
Example~\ref{psq}.

Check that the dimensions add up:
\[\dim_\Q(H(\Gamma_0(121),E)) = 29-13=16=14+2.\]

\bibliographystyle{amsalpha}
\bibliography{biblio}

\end{document}